\documentclass[12pt,leqno]{amsart}
\usepackage{amsmath, amssymb, amscd, amsfonts, xy,  stmaryrd, turnstile}%color

\usepackage%[colorlinks=true, pdfstartview=FitV, linkcolor=red, citecolor=blue, urlcolor=darkblue]
{hyperref}

\def\car#1,#2,#3,#4{
$$
   \CD
   #1           @>{}>>          #2        \\
   @V{{}}VV                  @VV{{}}V  \\
   #3        @>{{}}>>   #4
   \endCD
   $$}

\newtheorem{theorem}{Theorem}[section]
\newtheorem{lemma}[theorem]{Lemma}
\newtheorem{proposition}[theorem]{Proposition}
\newtheorem{corollary}[theorem]{Corollary}

\theoremstyle{definition}
\newtheorem{definition}[theorem]{Definition}
\newtheorem{example}[theorem]{Example}
\newtheorem{remark}[theorem]{Remark}

\begin{document}

\title[\'etale FIP extensions]{\'etale   extensions with finitely many subextensions}

\author[G. Picavet and M. Picavet]{Gabriel Picavet and Martine Picavet-L'Hermitte}
\address{Universit\'e Blaise Pascal \\
Laboratoire de Math\'ematiques\\ UMR6620 CNRS \\ Les C\'ezeaux \\ 24, avenue des Landais\\
BP 80026 \\ 63177 Aubi\`ere CEDEX \\ France}

\email{Gabriel.Picavet@math.univ-bpclermont.fr} 
\email{picavet.gm(at)wanadoo.fr}
\email{Martine.Picavet@math.univ-bpclermont.fr}

\begin{abstract}  The aim of this paper is the characterization of finite \'etale (unramified) ring extensions $R\subseteq S$ that have finitely many $R$-subalgebras.  It generalizes our earlier results on diagonal extensions of the type $R\subseteq R^n$, where $R$ is a ring. Results vary along the subextensions appearing into the canonical decomposition of an integral extension defined by the seminormalization and the $t$-closure of the extension. A special treatment is given for a subintegral  extension   $R\subseteq S$ , which is never unramified if $R\neq S$.
\end{abstract}

\subjclass[2010]{Primary:13B02, 13B21, 13B22; Secondary: 13E10, 13B30, 13B40, 05A18, 13F10}

\keywords  {SPIR,   \'etale extension, unramified extension, decomposed, inert, ramified minimal extension, flat epimorphism, FIP, FCP extension, Gilmer extension, integral, infra-integral extension, integrally closed, Artinian ring, seminormal,  subintegral, support of a module}

\maketitle

\section{A summary of the paper and some  Notation}

 All rings  considered  are commutative, nonzero and unital; all morphisms of rings are unital.  A ring is  called {\it local} (respectively, {\it  semi-local})  if it has only one maximal ideal (respectively, finitely many maximal ideals).   The  conductor of a (ring) extension $R\subseteq S$  is denoted by $(R:S)$ and the set of all $R$-subalgebras of $S$   by $[R,S]$. The extension $R\subseteq S$ is said to have FIP (for the ``finitely many intermediate algebras property") if $[R,S]$ is finite. A {\it chain} of $R$-subalgebras of $S$ is a set of elements of $[R,S]$ that are pairwise comparable with respect to inclusion. We say that the extension $R\subseteq S$ has FCP (for the``finite chain property") if each chain of $[R,S]$ is finite.  Moreover, an extension $R\subseteq S$ is dubbed {\it chained} if $[R,S]$ is a chain. 
 
 If $\mathcal M$ is a property of modules, we say that a ring extersion  $R\subseteq S$ is a $\mathcal M$ extension if  $\mathcal M$  holds for the $R$-module $S$ (for example, $\mathcal M$ may be the finite, free, projective, finite presentation properties). An extension  $R \subseteq S$ is said of finite type (finite presentation) if the $R$-algebra $S$ is of finite type (finite presentation). Epimorphisms are here   epimorphisms of the category of commutative rings. For example, $R\to R_\Sigma$ is a flat epimorphism for any multiplicatively closed subset $\Sigma$ of $R$.
 
  We say that a property $\mathcal P$ of ring morphisms $f: R\to S$ is {\it local } if $ R_P\to S_Q$ verifies $\mathcal P$  for each $Q\in \mathrm{Spec}(S)$ and $P:= f^{-1}(Q)$. This definition should not be confused with a ring morphism verifying locally a property, that is $R_P\subseteq S_P$ verifies $\mathcal P$ for each $P\in \mathrm{Spec}(R)$. If $R$  is a ring and $P\in \mathrm{Spec}(R)$, we denote by $\kappa (P)$ the residual field  $R_P/PR_P$ (of $R$ at $P$).  For a ring extension $R\subseteq S$, $Q\in \mathrm{Spec}(S)$ and $P:= Q\cap R$, there is a residual field extension $\kappa(P) \subseteq \kappa(Q)$.

  Clearly, an extension $R\subseteq S$ has FIP if and only if $R/(R:S) \subseteq S/(R:S)$ has FIP. So  we can assume in many case that the conductor is $0$. Moreover, recall that an integral extension $R\subseteq S$ has FCP if and only if $R/(R:S)$ is an Artinian ring and $R \subseteq  S$ is  finite \cite[Proposition 4.2]{DPP2}. Therefore, we can often assume that  $R$ is Artinian. The characterization of  integral extensions  that have FCP, first  appeared in the seminal R. Gilmer's paper \cite{G} on integral domains. As this condition is ubiquitous in this paper, we dub these extensions $R\subseteq S$ as follows, where $R/(R:S)$ is called the {\it nucleus} of $R \subseteq S$.
  
  \begin{definition}\label{defG} An extension $R\subseteq S$ is called a {\it Gilmer extension} if its nucleus is an Artinian ring.  Therefore, an integral  extension, having FIP, is a finite Gilmer  extension.
 \end{definition}
  It is easy to show that if $R\subseteq S$ is a finite Gilmer extension and $T\in [R,S]$, then $R\subseteq T$ and $T\subseteq S$ are finite Gilmer extensions.

 We will show that an unramified  Gilmer extension is finite, so that we are able to replace  integrality hypotheses with   finiteness assumptions.
  
 We now sum up our main results, using standard notation and definitions. Nevertheless, the reader is advised that we will precise them  later in this section. In an earlier paper, we gave a characterization  of {\it diagonal ring extensions }$R \subset R^n$ having FIP, $n> 0$ an integer,  involving  
 rings with finitely many ideals (see Theorem~\ref{PROD}). Such ring extensions are \'etale and finite and  actually  are trivial \'etale covering. Separable finite field extensions are also \'etale. They are known to have FIP because they are monogenic.  Monogenic properties are beyond the scope of this paper. Actually,  we characterized arbitrary FCP and FIP extensions  in \cite{DPP2}, a joint paper by D. E. Dobbs and ourselves, whose proofs do not need some monogenic conditions.  
  Another  point of view is given by  the so-called {\it separable} algebras explored in \cite{DMI}. We will not use the terminology of \cite{DMI}, because it may cause some confusion. Moreover in the context  of the paper, separable ring extensions identify to {\it \'etale} algebras. These observations motivated us to examine  integral  \'etale extensions that have FIP. Actually, we will  also look at integral unramified extensions, because on one hand the unramified property is sufficient to get some results, and on the other hand,  the flatness condition acts on the nullity of the conductor. Note also that we have a surprising converse:   if $R\subset S$ is a seminormal and infra-integral finite Gilmer extension, then $R\subset S$ is unramified and has FIP (see Proposition~\ref{5.5}).

 As in some  of our earlier papers, our  strategy  will consist to use the canonical decomposition of  an integral ring extension. At this stage some recalls  about seminormality and $t$-closedness are needed. Following \cite{Pic 2} and \cite{S}, an  extension  $ R\subseteq S$  is termed:
\begin{enumerate}
\item  {\it infra-integral}  if $R\subseteq S$ is integral and  all its residual extensions  are isomorphisms.  

\item  {\it subintegral} if $R\subseteq S$ is infra-integral and spectrally bijective. 

\item  {\it seminormal} if the relations $b\in S,\ b^2,b^3\in R \Rightarrow b\in R$. 
 \item {\it $t$-closed} if for all  $b\in S, r\in R$, $(b^2-rb,b^3-rb^2\in R )\Rightarrow b\in R$. 
 \end{enumerate}
  The {\it seminormalization}  $\substack{+\\ S}R$    of $R$ in $S$  is the greatest subalgebra $B\in[R,S]$ such that $R\subseteq B$ is  subintegral and the smallest subalgebra $C\in [R,S]$ such that $C\subseteq S$ is seminormal. 
 
   The {\it $t$-closure}  $\substack{t\\ S}R$    of $R$ in $S$  is the greatest subalgebra $B\in[R,S]$ such that $R\subseteq B$ is  infra-integral and the smallest subalgebra $C\in [R,S]$ such that $C\subseteq S$ is $t$-closed.

For an integral extension $R\subseteq S$, the tower $R\subseteq \substack{+\\ S}R\subseteq \substack{t\\ S}R \subseteq S$ is called the {\it canonical decomposition} of $R\subseteq S$.

Seminormalizations  and $t$-closures  commute with the formation of  localizations at arbitrary multiplicatively closed subsets   \cite[Proposition 2.9]{S} and \cite[Proposition 3.6]{Pic 1}. If $R\subset S$ is seminormal, $(R:S)$ is a radical ideal of $S$. 

It  is now worth noticing that there are integral extensions having FIP, which  are not unramified.  We studied  in \cite{Pic 5}  idealizations having FIP. Let $R$ be a ring and $M$ an $R$-module, then the {\it idealization} of $M$ is the extension $R\to R\times M$ defined by  $r\mapsto (r,0)$, while the multiplication on $R\times M$ is defined by $(r,m)(s,n)= (rs,rn+sm)$. We proved that an idealization of an $R$-module $M$ has FIP if $M$ has finitely many submodules.  Such an extension is subintegral but is never unramified (see Remark~\ref{NOUN}(2)).

In Section 2  we give  recalls and technical results needed for our theory. In particular, we introduce minimal extensions  and their relations with the FIP and FCP properties.  A crucial result is that a quasi-finite (for example unramified) Gilmer extension is  finite. We also introduce some special rings like SPIRs and rings with finitely many ideals. For a ring extension $R \subseteq S$ with finite maximal support, we define the localization  $R_{(R:S)}\subseteq S_{(R:S)}$ with respect to $(R:S)$, where $R_{(R:S)}$ is a semi-local ring. This localization encodes many properties  of $R\subseteq S$, while its conductor is zero when $R\subseteq S$ is finite and flat, for example  finite \'etale.

Section 3 develops global results, apart some localization results about \'etale Gilmer extensions. See Theorem~\ref{2.4} and Corollary~\ref{2.5}, where we prove in particular that a Gilmer \'etale extension is locally a trivial \'etale cover. Our first result shows that a $t$-closed unramified finite extension $R\subset S$  has FIP if and only if $R\subset S$ is a Gilmer extension. Many results ensue from the fact that some \'etale covers  become (locally) trivial after a base change.  It is then enough to use our earlier results on trivial \'etale covers in \cite{Pic 4}. For example, an \'etale extension has FIP when its domain is a FMIR and the extension verifies some rank properties (see Theorem~\ref{2.9}). Also an unramified Gilmer extension $R\subset S$ has FIP, when its nucleus is a  reduced ring. We will see that the subintegral part $R\subseteq  \substack{+\\ S}R$ is never unramified, even if $R\subset S$ is unramified. This fact introduces some real difficulties that we already met  in \cite{DPP3} and \cite{Pic 6}.  One main result is  that if $R\subseteq S$ is an \'etale extension, whose domain has only finitely many ideals, then  $R\subset S$ has FIP, in case  $\mathrm{rk}_P(S_P) \leq 2$ for each  $P\in \mathrm{Supp}(S/R)$, such that $R_P$ is a SPIR.  Naturally, the converse  does not hold, as the following caricatured example shows. Whatever may be  a ring $R$, the extension $R\subseteq R$ is \'etale and has FIP!
The above result generalizes Theorem~\ref{PROD} gotten for diagonal extensions.

In Section 4, we suppose that $R$ is a local ring in order to study the subintegral part
of a finite extension $R \subset S$. Indeed this part is never unramified even if $R\subseteq S$ is \'etale finite. So this section may seem to be disconnected from the subject of the paper but is here to give a description as complete as possible.  Actually we show that a subintegral FIP extension $R\subset T$  over a local ring  can  be immersed in  an \'etale extension, necessarily of the type $R\subset R^n $ with $R$ non-reduced if and only if $T$ verifies some conditions (see Proposition~\ref{4.42}).  It remains that we are not able to give a characterization of  the subintegral part of a finite unramified (\'etale) extension. As a by-product we get that the FIP property can be reduced to local subextensions when $R$ is  Henselian (Artinian). Under additional (and necessary) hypotheses we get that a subintegral ring extension has FIP if and only if it is chained. A significative consequence is as follows. Let $R\subset S$ be a finite \'etale  local  Gilmer extension which is not seminormal  and whose nucleus is an infinite Artinian ring.  Then $R\subset S$ has FIP if and only if $R\subset \substack{+\\ S}R$ is chained (see Proposition~\ref{4.10}).  

Section 5 is concerned with seminormal extensions. We show that a finite seminormal and infra-integral Gilmer extension $R\subset S$  is unramified and has FIP. This  case corresponds to the part $\substack{+\\ S}R \subseteq \substack{t\\ S}R$ of the canonical decomposition. The behavior of   $\substack{t\\ S}R \subset S$ is elucidated  in Section 3. 
The above considerations show that the seminormal case is much more agreeable and gives rise to  results different from the previous sections. We also  get that $R\subset S$ has FIP and is integral seminormal if and only if its nucleus is Artinian and reduced, among other results characterizing seminormal finite FIP extensions.

\section{Recalls and results needed in the sequel}

\subsection{Unramified and \'etale algebras}
 We recall some material about \'etale algebras.
 A ring morphism $R \to S$ is called  {\it unramified, (net in French, with the Raynaud's definition \cite{R})}  if $R\to S$ is of finite type and its $S$-module of {\it K\"ahler}  differentials $\Omega (S|R) = 0$. An $R$-algebra $S$ of finite type is unramified if and only if the Property UN holds (see  \cite[Exercices p. 38]{R}(1)).
 
  {\bf UN}:   $PS_Q = QS_Q$ and $\kappa (P) \to \kappa (Q)$ is a finite   and separable field extension for all $Q\in \mathrm{Spec}(S)$ and  $P:= Q\cap R$.

 We will use the following definition. A ring  morphism $R\to S$ is  called {\it \'etale}  if and only if it  is flat of finite presentation and unramified \cite[Corollaire 1 p.55]{R}. The  \'etale and unramified properties are universal, that is, stable under any base change.
 
 A ring morphism $R\to S$ is called an {\it \'etale cover} if $R\to S$ is \'etale and finite. An \'etale cover is called {\it trivial} if it is (isomorphic to) a diagonal extension  of the form $R\subset R^n$ for some integer $n>0$  \cite[Remarque 18.2.7]{EGAIV}.
 
 \begin{proposition}\label{DUN} \cite[Proposition 17.3.3(v)]{EGAIV} Let $R\subset S$ be an unramified  extension and $T\in [R,S]$, then $T\subseteq S$ is unramified.
 \end{proposition}

 We note here for later use the following result.
 \begin{lemma}\label{Omega} Let $R\subset S$ be an  extension, such that $R$ and $S$ share an ideal  $I$. Then  $\Omega (S|R)$ is isomorphic to the $S$-module $\Omega ((S/I)|(R/I))$, gotten  via $S\to S/I$.  
\end{lemma} 
\begin{proof} First observe that $\Omega (S|R) = J/J^2$ where $J=\ker(S\otimes_RS \to S)$ is generated over $S$ by the elements $s\otimes 1 -1\otimes s$, with $s$ in $S$. It is clear that $IJ= 0$; so that $J/J^2$ is an ($S/I$)-module. We thus get $\Omega (S|R) \cong \Omega (S|R)\otimes_{(S/I)}(S/I) \cong \Omega (S|R)\otimes_S(S/I) \cong \Omega((S/I)|(R/I))$  by \cite[Proposition 5, p.27]{R}. 
\end{proof}

 \subsection{Some finiteness results}
  Next  result is crucial.

  \begin{proposition}\label{SE} \cite[Proposition 2, p.6-02]{O} Let $R\to S$ be a flat epimorphism where  $R$ is zero-dimensional. Then $R\to S$ has the property $\mathbf{SE}$: $R\to S$ is surjective.
\end{proposition}
 
The following result is a strengthening of a well-known result \cite[Proposition 18.3.1]{EGAIV}. It provides a connexion with the work of \cite{DMI}.
 
 \begin{lemma}\label{Seppro} Let $R\subseteq S$ be a finite  extension.
 \begin{enumerate}
 
 \item  $R\subseteq S$ is of finite presentation if $R\subseteq S$ is projective.
 
   \item  $R\subseteq S$  and $S\otimes_RS \to S$ are projective if and only if $R\subseteq S$ is projective and unramified, and if and only  if $R\subseteq S$ is \'etale. 
   
   \item If $R\subseteq S$ is  \'etale and $R\subseteq T$ is an intermediary unramified extension, then $R\subseteq T$ and $T\subseteq S$ are \'etale,  projective and finite.
   \end{enumerate}
\end{lemma}
\begin{proof}
(1) Since $R\subseteq S$ is finite, it is of finite type. Then use  \cite[Theorem 4.4]{OR}.

(2)  Suppose that $R\subseteq S$ is \'etale. Since $R\subseteq S$ is of finite presentation and finite, by \cite[Chapter I, Proposition 6.2.10]{EGA}, $S$ is an $R$-module of finite presentation. It follows that $S$ is a projective $R$-module because it is flat. For $S\otimes_R S\to S$, see  \cite[Proposition 18.3.1]{EGAIV}. For the converse, use (1) to get that $R\subseteq S$ is of finite presentation. Then again apply  \cite[Proposition 18.3.1]{EGAIV}.

(3) We first observe that  $T\subseteq S$ is unramified  by Proposition~\ref{DUN}. Consequently,  $T\subseteq S$ is projective because $R\subseteq T$ is unramified \cite[Chapter II, Proposition 2.3]{DMI}, whence faithfully flat. It is also of finite presentation again  by  \cite[Chapter I, Proposition 6.2.10]{EGA} or by (1). Therefore, $T\subseteq S$ is \'etale by \cite[17.3.4]{EGAIV} and so is $R\subseteq T$ by \cite[17.7.7]{EGAIV}. Moreover, $R\subseteq T$ is finite because it is integral.
\end{proof}

\begin{proposition}\label{GUF} Let $R\subseteq S$ be a quasi-finite  Gilmer extension, then $R\subseteq S$ is finite. In particular,  Gilmer extensions  that are either unramified or have FCP are finite. Moreover, an \'etale Gilmer extension is finite and projective. 
\end{proposition}
\begin{proof}  We can assume that $(R:S)= 0$ and that $R$ is an Artinian ring. The result is then  a consequence of the Zariski Main Theorem, since a quasi-finite extension $R\subseteq S$ is a finite extension $R\subseteq  R'$ followed by a flat epimorphism of finite presentation $R'\to S$ \cite[Corollaire 2, p.42]{R}. 
Then $R \to S$ is finite because $R' \to S$ is a flat epimorphism whose domain is zero-dimensional whence is surjective  by Property (SE). 
\end{proof}

 \begin{proposition}\label{RED} Let $R\subseteq S$ be an extension and $P\in \mathrm{Spec}(R)$.
  \begin{enumerate}
  
  \item If $R\subseteq S$ is \'etale and $R$ is reduced, then  $S$ is reduced. It follows that $PS$ is a radical ideal.
  
  \item If $R \subseteq S$ is a unramified  extension and $R$ is an Artinian ring, then $PS =Q_1\cap \cdots \cap Q_n$, where $Q_1,\ldots,Q_n$ are the prime ideals of $S$ lying over $P$. 
  
     \end{enumerate} 
   \end{proposition}
   \begin{proof} (1) is well-known \cite[Proposition 1, p.74]{R}. 
   
(2) Assume that $R\subseteq S$ is unramified and that $R$ is Artinian.  Then $R\subseteq S$ is finite by Proposition~\ref{GUF}. This entails that $S$ is an Artinian ring isomorphic to $\prod [S_Q \mid Q \in \mathrm{Spec}(S)]$. It follows that $PS = \prod [PS_Q \mid Q\cap R = P] = \prod[QS_Q \mid Q\cap R=P]$ by Property UN, so that $PS =Q_1\cap \cdots \cap Q_n$, where the prime ideals $Q_i$ are those above $P$.
 \end{proof}

\subsection{Minimal extensions and extensions having FCP}

 Minimal extensions  are examples of  extensions that have FIP. This concept was  introduced by Ferrand-Olivier \cite{FO}. Recall that an extension $R\subset S$ is called {\it minimal} if $[R,S]=\{R,S\}$. The key connection between the above ideas is that if $R\subseteq S$ has FCP, then any maximal (necessarily finite) chain $R=R_0\subset R_1\subset\cdots\subset R_{n-1}\subset R_n=S$, of $R$-subalgebras of $S$, with {\it length} $n<\infty$, results from juxtaposing $n$ minimal extensions $R_i\subset R_{i+1},\ 0\leq i\leq n-1$.  Since minimal extensions are of finite type (monogenic), it follows that an extension that has FIP is of finite type. We will use the following theorem.
 
  \begin{theorem}\label{MIN} \cite[Theorem 3.3]{Pic} Let $R\subset T$ be a ring extension and $M: =(R:T)$. Then $R\subset T$ is minimal and finite if and only if $M\in\mathrm{Max}(R)$ and one of the following three conditions holds:
  
\item[(1)]  {\bf inert case}: $M\in\mathrm{Max}(T)$ and $R/M\to T/M$ is a minimal field extension;

\item[(2)] {\bf decomposed case:} There exist $M_1,M_2\in\mathrm{Max}(T)$ such that $M= M_1 \cap M_2$ and the natural maps $R/M\to T/M_1$ and $R/M\to T/M_2$ are both isomorphisms;

\item[(3)]  {\bf ramified case}: There exists $M'\in\mathrm{Max}(T)$ such that ${M'}^ 2 \subseteq M\subset M',  [T/M:R/M]=2$, and the natural map $R/M\to T/M'$ is an isomorphism.

Then $\mathrm{Supp}(T/R) = \{M\}$ holds in each of the above three cases.
Moreover, $R/M\subset T/M$ identifies to $R/M \subset R/M\times R/M$ in case $\mathrm{(2)}$ and  to $R/M\subset(R/M)[X]/(X^2)$ in case $\mathrm{(3)}$. 
\end{theorem}

Next result describes minimal extensions appearing in the canonical decomposition.  If $R\subset U \subset V \subset S$ is a composite of extensions, then $U\subset V$ is called a {\it subextension} of $R\subset S$.

\begin{theorem}\label{CANMIN}  Let $R\subseteq S$ be an integral extension that has FIP, with canonical decomposition $R\subseteq \substack{+\\ S}R\subseteq \substack{t\\ S}R \subseteq S$. 
\begin{enumerate}
 \item $R\subseteq \substack{+\\ S}R$  is a composite of finitely many  ramified minimal extensions and its minimal subextensions  are ramified. 
  \item $\substack{+\\ S}R\subseteq \substack{t\\ S}R$  is a composite of finitely many  decomposed minimal  extensions and its minimal subextensions  are decomposed 

 \item $ \substack{t\\ S}R \subseteq S$ is a composite of finitely many  inert minimal extensions and its minimal subextensions  are inert.
 \end{enumerate} 
\end{theorem}
\begin{proof} (1) Since $R\subseteq \substack{+\\ S}R$ is subintegral, its spectral map is bijective, and its residual extensions are isomorphisms. These properties hold for any minimal subextension, which is therefore ramified.

(2) Since $\substack{+\\ S}R\subseteq \substack{t\\ S}R$ is infra-integral, its residual extensions are isomorphisms. This property holds for any minimal subextension, which has to be either ramified or decomposed. Let $A\subset B$ be such a minimal subextension and set $M:=(A:B)$. Since $\substack{+\\ S}R\subseteq \substack{t\\ S}R$ is seminormal, so is $\substack{+\\ S}R\subseteq B$. Then, $C:=(\substack{+\\ S}R: B)$ is a radical ideal in $B$. Moreover, $\substack{+\\ S}R\subseteq B$ is a finite FIP extension, so that $R/C$ and $B/C$ are Artinian rings, giving that $B/C$ is a product of finitely many fields. 
 
It follows that $B/C$ is absolutely flat. Then $C\subseteq M$ implies that $M/C$ is a radical ideal of $B/C$, so that $M$ is a radical ideal of $B$, and $A\subset B$ is decomposed. 

(3) is  \cite[Lemma 5.6]{DPP2}. 
\end{proof}

\begin{proposition}\label{MOmega} Suppose that $R\subset S$ is an integral minimal extension. Then $R\subset S$ is not unramified if it is ramified. If $R\subset S$ is  inert, it  is unramified if and only if $R/(R:S) \subset S/(R:S)$ is a separable field extension. If $R \subset S$ is   decomposed, it is unramified.
\end{proposition}
\begin{proof}  Use Theorem~\ref{MIN} which describes the extension $R/M \subset S/M$ where $M=(R:S)$ for a minimal extension $R\subset S$ with conductor a maximal ideal $M$. Then apply Lemma~\ref{Omega}.
\end{proof}

\begin{remark}\label{NOUN} (1) We immediately  observe that a finite minimal morphism  has FIP but  is not generally  \'etale. For example consider a finite minimal morphism $R \subseteq S$ whose domain is local and is not a field. By  Theorem~\ref{MIN}, its conductor is the maximal ideal of $R$. If  such a  morphism is \'etale, it is flat, so that $R$ is a field by \cite[Lemme 4.3.1]{FO}. 

(2)  We now observe that a subintegral extension $R\subset S$, that is either finite or of finite type, is never unramified. Deny, then Zorn Lemma guaranties that there is some $T\in [R,S]$ such that $T\subset S$ is a minimal extension and this extension is unramified and finite. Let $M\in \mathrm{Max}(T)$ be  the conductor $(T:S)$. Then $T/M \subseteq S/M$ is also unramified. If $T\subset S$ is inert, then $T/M= S/M$, an absurdity. If $T\subset S$ is decomposed, the spectral injectivity is violated. If $T\subset S$ is ramified, then $S/M\cong (T/M)[X]/(X^2)$ is not unramified over $T/M$, because the polynomial $X^2$ is not separable.
\end{remark}
 
  As in some of our earlier papers,  the canonical decomposition of a finite extension $R\subset S$ leads to different cases. The subintegral part $R \subseteq \substack{+\\ S}R$ cannot be unramified  by Remark~\ref{NOUN}(2), even if $R\subset S$ is unramified. This fact forbids us to use Lemma~\ref{Seppro}, except when $R$ is seminormal in $S$. This case is studied in Section 5.

 \subsection{Some special rings}

 Rings with finitely many ideals were characterized by D. D. Anderson and  S. Chun.
Recall that a SPIR is a {\it special principal ideal ring}, {\it i.e.} a ring $R$ with a unique nonzero prime ideal $M=Rt$, such that $M$ is nilpotent of index $p>0$. Hence a SPIR is not a field and each nonzero element of a SPIR is of the form $ut^k$ for some unit $u$ and some {\it unique} integer $k< p$. We deduce from the proof of \cite[Lemma 5.11]{DPP2} that an Artinian local  ring $(R,M)$ is infinite if and only if $R/M$ is infinite.

\begin{theorem} \label{FMIR}\cite[Corollary 2.4]{AC} A ring $R$ has only finitely many ideals if and only if $R$ is (isomorphic to) a finite product of finite local rings, infinite SPIRs and fields. We call  FMIR such a ring.
\end{theorem}

We observe that a  FMIR  is an Artinian ring, and  that its local rings can only be infinite fields, finite rings and infinite SPIRs. As  we wrote in the summary, our interest was motivated by the two next statements.
FMIRs already appear  when considering \'etale algebras over a field.

\begin{theorem}\label{PROD}\cite[Theorem 4.2]{Pic 4} Let $R$ be a ring and $n>1$  an integer. Then $R\subseteq R^n$ has FIP if and only if $R$ is a FMIR,  with $n=2$ when $R$ has at least an infinite SPIR local ring.
\end{theorem}

 \begin{theorem}\label{ET} \cite[Proposition 3, A V 29]{Bki A} If $K$ is a field and $K\subseteq A$ is an \'etale algebra, then $K\subseteq A$ has FIP and $A$ is a  FMIR. 
\end{theorem}

As an example, we extract from  \cite[Proposition 4.15]{DPP3} the following more precise  result. For a positive integer $n> 1$,  recall that the $n$th {\it Bell number} $B_n$ is the number  of partitions of $\{1,\ldots,n\}$ \cite[p. 214]{An}. 

\begin{proposition}\label{1.5} Let $R$ be ring, then the diagonal extension $R \subseteq R^n$ is \'etale (a trivial \'etale cover), for any integer $n>1$. Moreover, if $R$ is a  field, then
$[R,R^n]$ has $B_n$ elements and $R^n$ has $2^n$ ideals.
 \end{proposition}

\subsection{Some results on extensions having FIP}
Here are some notation, definitions and results needed in the sequel.

Let $R$ be a ring. As usual, Spec$(R)$ (resp$.$ Max$(R)$) denotes the set of all prime ideals (resp$.$ maximal ideals) of $R$. If $I$ is an ideal of $R$, we set ${\mathrm V}_R(I):=\{P\in\mathrm{Spec}(R)\mid I\subseteq P\}$ and $\mathrm{D}_R(I)$ is its complement. If $R\subseteq S$ is a ring extension and $P\in\mathrm{Spec}(R)$, then $S_P$ is the localization $S_{R\setminus P}$. We denote the integral closure of $R$ in $S$ by $\overline R$. Recall that if $E$ is an $R$-module, its {\it support} $\mathrm{Supp}_R(E)$ is the set of prime ideals $P$ of $R$ such that $E_P\neq 0$ and $\mathrm{MSupp}_R(E):=\mathrm{Supp}_R(E)\cap\mathrm{Max}(R)$ is also the set of maximal elements of $\mathrm{Supp}_R(S/R)$, because a support is stable under specializations.  
Finally,  $\subset$ denotes proper inclusion and $|X|$ denotes the cardinality of a set $X$.

\begin{proposition}\label{SUP} Let $R\subseteq S$ be a ring extension.

\item [(1)] If $R\subseteq S$ has FCP (FIP), then  $|\mathrm {Supp}(S/R)|<\infty$.

\item [(2)]If  $|\mathrm {MSupp}(S/R)|<\infty$, then $R\subseteq S$ has FCP (FIP) if and only if $R_M\subseteq S_M$ has FCP (FIP) for each $M\in\mathrm {MSupp}(S/R)$.
\item[(3)] If $R \subseteq S$ is integral, then $R\subseteq S$ has FCP if and only if $R\subseteq S$ is a  finite Gilmer extension.
\item[(4)]  $R\subseteq S$ has FIP if and only if  $R\subseteq \substack{+\\ S}R$, $\substack{+\\ S}R\subseteq \substack{t\\ S}R$ and $ \substack{t\\ S}R \subseteq S$ have FIP.
\end{proposition}

\begin{proof}  Read \cite[Proposition 3.7,  Corollary 3.2, Theorem 3.13,  Theorem 4.2, Theorem 5.9]{DPP2}.
\end{proof}

\begin{corollary}\label{SUPF}  If $R\subseteq S$ is an integral extension that has FIP, then $R\subseteq S$ is finite and a composite of $n$  finite  minimal  extensions  $R_{i} \subset R_{i+1}$ with conductor $M_i$. Moreover, $\mathrm {Supp}_R(S/R)$ is a finite set; in fact, $\mathrm{V}_R((R:S))= \mathrm {Supp}_R(S/R)=\{M_i\cap R\mid i=0,\ldots,n-1\}\subseteq \mathrm{Max}(R)$.
\end{corollary}
\begin{proof} Read \cite[Corollary 3.2]{DPP2}, \cite[Proposition 17, p. 133]{Bki AC}.
\end{proof}

We note here two useful results.

\begin{proposition}\label{QUOT} Let $R\subseteq S$ be an extension that has FIP.  Let also $J$ be an ideal of $S$ and $I:=J\cap R$. Then $R/I \subseteq S/J$ has FIP.
\end{proposition}
\begin{proof} It is enough to observe that an element of $[R/I,S/J]$ is of the form $T/J$ where $T\in [R,S]$ contains $J$. By a Noether  theorem, we have $T/J= (T+J)/J\cong T/(J\cap T)$.
\end{proof}

\begin{remark}\label{LOCPROD} We will also use the following result. If $R_1,\ldots,R_n$ are finitely many rings,  the ring $R_1\times \cdots \times R_n$ localized at the prime ideal $P_1\times R_2\times\cdots \times R_n$ is isomorphic to $(R_1)_{P_1}$ for $P_1 \in \mathrm{Spec}(R_1)$. This rule works for any other prime ideal of the product.  
\end{remark}

\subsection{Localization with respect to the conductor}

The above results suggest  to introduce the following definition.
\begin{definition}\label{SEM} Let $R\subset S$ be a ring extension   such that $\mathrm{MSupp}(S/R) $ has finitely many elements $M_1,\ldots,M_n$. For example, $\mathrm{MSupp}(S/R) = \mathrm{Supp}(S/R)$ is finite if $R\subseteq S$ is a finite Gilmer extension.

 We set 
$R_{(R:S)} := R_{R\setminus (M_1\cup \cdots \cup M_n)}$ and $ U_{(R:S)}:= U_{R\setminus(M_1\cup \cdots\cup M_n)}\cong U\otimes_RR_{(R:S)}$ for any $R$-module $U$. 
The ring $R_{(R:S)}$ is semi-local.
\end{definition}
Note that  $ 1+ (R:S) \subseteq R\setminus ( M_1\cup \cdots\cup M_n)$.  In case $\mathrm{V}_R((R:S))= \mathrm {Supp}_R(S/R)$, (for example if $R\subseteq S$ is finite) the above multiplicatively closed subsets have the same saturation, so  that $R_{1+ (R:S)} \to R_{(R:S)}$ is an isomorphism.

In case $R\subseteq S$ is finite and $\mathrm{MSupp}(S/R)$ is finite, we have:

 \noindent  $(R_{(R:S)}: S_{(R:S)}) = (R:S)_{ (R:S)}$ and    $\mathrm{MSupp}(S_{(R:S)}/R_{(R:S)}) = \mathrm{Max}(R_{(R:S)})$.

The extensions $R_{(R:S)} \subset  S_{(R:S)}$ and $R\subset S$, with the hypotheses of Definition~\ref{SEM}, share  usual properties. 

\begin{proposition}\label{PROPN} Let $R \subseteq S$ be an  extension such that $|\mathrm{MSupp}(S/R)|$ $<\infty$. Then $R_{(R:S)} \subseteq  S_{(R:S)}$  verifies $\mathcal P$ if and  only if $R\subseteq S$ verifies $\mathcal P$, when $\mathcal P$ is one of the following properties:  finite, of finite type, Gilmer property, FIP, flatness, unramified,
 seminormal, infra-integral, subintegral. Note  that $R= S\Leftrightarrow R_{(R:S)}=S_{(R:S)}$.
\end{proposition}
\begin{proof} We first observe that  a maximal ideal of $R$ is either in $\mathrm{Supp}(S/R)$ or not. Moreover, the conductor of a finite extension commutes to the formation of localizations. Finiteness properties hold because of Lemma~\ref{FIN}. For the FIP property, see Proposition~\ref{SUP}(2). For the Gilmer property, it is enough to observe that a ring $R$ is Artinian if and only if   $\mathrm{Max}(R)$ is finite and the length of $R_M$ is finite for each $M\in \mathrm{Max}(R)$. For some other properties, we may use the following fact.  By the very definition of a support, $R_P=S_P$ for any  $P\notin \mathrm{Supp}_R(S/R)$ and also $R_P= S_Q$  for each $Q\in  \mathrm{Spec}(S) $ lying over $P$. Therefore,  $R\subseteq S$ is flat (respectively, unramified) if and only if so is $R_{(R:S)} \subseteq  S_{(R:S)}$.   
For the seminormal, infra-integral and subintegral properties, use the following statement. Let  $R \subseteq S$ be an extension. Then the map $\mathrm{D}_S((R:S)) \to \mathrm{D}_R((R:S))$ induced by  $\mathrm{Spec}(S) \to \mathrm{Spec}(R)$ is an homeomorphism and $R_P\to S_Q$ is an isomorphism for each $Q\in \mathrm{D}_S((R:S))$ and $P:= Q\cap R$.
 \end{proof}
 
 \begin{remark}\label{LOCAL} (1) Let $R\subseteq S$ be an extension such that $R_{(R:S)} \subseteq S_{(R:S)}$ is defined  and $\mathcal P$ a property of ring morphisms verified by ring isomorphisms. Then $R\subseteq S$ verifies locally $\mathcal P$ if and only if $R_{(R:S)}\subseteq S_{(R:S)}$ verifies locally $\mathcal P$.
 
 (2) We can add in Proposition~\ref{PROPN}  two properties $\mathcal P$: ``finite and projective'' and ``finite and \'etale'' when $R\subseteq S$ is of finite presentation (see Theorem~\ref{2.51}). This holds if $R$ is a ring such that flat modules of finite type are projective {\em i.e} is a $ \mathcal S$-ring. Known cases of $\mathcal S$-rings are semi-local rings, and rings whose total quotient ring is semi-local, as Noetherian rings and integral domains (see \cite[Theorem 2]{E}). 
 \end{remark}
 
\begin{lemma}\label{FIN} Let $R \subseteq S$ be a ring extension such that $|\mathrm{MSupp}(S/R)| < \infty$. Then $R_M\subseteq S_M$ is finite (of finite type) for each $M\in \mathrm{MSupp}(S/R)$ if and only if $R \subseteq S$ is finite (of finite type).
\end{lemma}

\begin{proof}  Let  $\phi_M: S\to S_M$ be the natural map for  $M \in \mathrm{MSupp}(S/R)$ and assume that $R_M\subseteq S_M$ is finite for each $M\in \mathrm{MSupp}(S/R)$. There is a finite system $\mathcal G_M$ of  elements of $S$  such that $\phi_M(\mathcal G_M)$  generates $S_M$ over $R_M$. Now $\cup [\mathcal G_M \mid M \in \mathrm{MSupp}(S/R)]$ clearly generates $S$ over $R$ as a module. For the finite type property, the proof is similar.
\end{proof}

The following definition will also be used.

\begin{definition}\label{SEMINF} Let $R\subseteq S$ be an extension with $|\mathrm{MSupp}(S/R)|< \infty$. We set $\mathrm{MSupp}^* (S/R): = \{M \in \mathrm{MSupp}(S/R) \mid |R_M| = \infty\}$. If $U$ is any $R$-module, we also set $U_{(R:S)^*} = U_{R\setminus  (M_1\cupÊ\cdots \cup M_p)}$, where $\mathrm{MSupp}^*(S/R) = \{M_1,\dots, M_p\}$.
\end{definition}
\begin{proposition}\label{FIPINF} Let $R\subseteq S$ be a finite extension, with $|\mathrm{MSupp}(S/R)|$

\noindent $< \infty$. Then $R\subseteq S$ has FIP (respectively, is a Gilmer extension) if and only if $R_{(R:S)^*} \subseteq S_{(R:S)^*}$ has FIP (respectively, is a Gilmer extension).
\end{proposition}
\begin{proof} If $R_M$ is finite, then $R_M\subseteq S_M$  has FIP, because $R\subseteq S$ is finite. Then use Proposition~\ref{SUP}(2) for the FIP property. For the Gilmer property, observe that the length of $R_M$ is finite if $|R_M|$ is finite.
\end{proof}

\begin{lemma}\label{CONZERO} If $R\subset S$ is a  finite and locally free ({\it i.e} flat)  extension, then $(R_P:S_P)={(R:S)}_P= 0$, for each $P \in \mathrm{Supp}(S/R)$. If in addition $R\subset S$ is a Gilmer extension, then $(R_{(R:S)}:S_{(R:S)})=0$, $(R:S) =\ker (R \to R_{(R:S)})$  and $R/(R:S) \cong R_{(R:S)}$.
\end{lemma}
\begin{proof} Since $R\subset S$ is finite, we can reduce to the case where $(R,M)$ is local and then $R\subset S$ is free (of rank $n$). Let $\{e_1,\ldots ,e_n \}$ be a basis of $S$ over $R$.  Suppose that $n> 1$, then we can write $1=  \sum r_ie_i$ with $r_i\in R$. If each $r_i$ is in $M,$ then $1$ belongs to some maximal ideal $N$ of $S$ lying over $M$, because $R\subseteq S$ is finite. This is absurd and we can suppose that $r_1$ is  a unit of $R$. Therefore, $e_1= r_1^{-1}(1-\sum_{j\neq 1} r_je_j)$ entails that $\{1,e_2,\ldots,e_n\}$ is a basis of $S$ over $R$. Now for $r\in(R:S)$ we have $re_2= u1$ for some $u\in R$, so that $r= 0$. If $n= 1$,  let $\{e_1\}$ be a basis, then $1=re_1$ for some $r\in R$ shows that $r$ is a unit in $S$. As $R\subset S$ is integral and injective, we see that    $r$ is invertible in $R$ and   then $e_1\in R$. It follows that  $R= S$, an absurdity.

We observe that $I:=(R:S)$ is locally trivial and, in particular is a pure  ideal (such that $R/I$ is flat over $R$). It follows that $I\cap J =IJ$ for each ideal $J$ of $R$ and then for all $x \in I$, there is some $y\in I$, such that $x=xy$. Therefore, $I$ is the kernel of $R \to R_{1+I} \cong R_{(R:S)}$. Then $R/I \to R_{(R:S)}$ is a flat epimorphism by \cite[Corollaire 3.2, p.114]{L}(i). Because maximal ideals of $R/I$ can be lifted up to $R_{(R:S)}$, this last map is an isomorphism.
\end{proof}

\section{\'etale morphisms}

We begin with some results on diagonal extensions  $R \subseteq R^n$, $n\geq 2$ an integer. 
\subsection{Properties of diagonal extensions}

\begin{lemma}\label{2.1}  For $n\geq 2$, $(R:R^n) = 0$,  $R \subseteq R^n$ is locally an isomorphism, whence is infra-integral. Moreover, $R\subseteq R^n$ is seminormal if and only if $R$ is reduced.
\end{lemma}
\begin{proof}   
The extension is  a local isomorphism and infra-integral  by Remark~\ref{LOCPROD}. 

Suppose that $R$ is reduced. If $x, y \in R$ are such that $x^2=y^2$ and $x^3=y^3$, then $x=y$ by \cite[Lemma 3.1]{S}. It follows that $R\subseteq R^n$ is seminormal. Conversely, if $R \subseteq S$ is seminormal, its conductor  $0$ is a radical ideal and $R$ is reduced.
\end{proof}

We will show that an \'etale Gilmer extension is infra-integral if and only if it  is locally a diagonal extension (see Corollary~\ref{2.5}(3)).

Note that $R \subseteq R^n$ cannot be subintegral since $n \geq 2 $.

Actually, we have a  result similar to Lemma~\ref{2.1} for an extension  which has FIP. 

\begin{proposition}\label{2.2}  If $R \subset S$ is an integral extension that has FIP, then $R\subset S$ is seminormal if and only if $(R:S)$ is a semiprime  ideal of $S$.
\end{proposition}
\begin{proof}
 We need only to show that $(R:S)$ semiprime in $S$ implies the seminormality of $R\subset S$. Consider the seminormalization $T$ of $R$ in $S$ and suppose that $R\neq T$. Then there exists some integral minimal extension $R\subset U$ with $U\subseteq T$. Then $R\subset U$ does be ramified with conductor $P$ a maximal ideal of $R$ such that there is some maximal ideal  $P'$ of $U$ such that $P'^2\subseteq P\subset P'$ because of Theorem~\ref{CANMIN} . Clearly, we have $(R:S)\cap U \subseteq (R:U) \subset P'$. Since $(R:S)$ is supposed to be semiprime in $S$, it follows that $(R:S)\cap U= P'\cap Q$, where $Q$ is either $U$ or an intersection of maximal ideals of $U$, different from $P'$. Therefore,  the $P'$-primary ideal $(R:U)$ contains  $P'Q$, whence is equal to $P'$, a contradiction.
 \end{proof}

 \subsection{First results}
 
 We prove a first positive result in the $t$-closed  case.

 \begin{proposition}\label{2.3} Let $R\subset S$ be a $t$-closed unramified  finite (integral) extension. Then $R\subset S$ has FIP if and only if $R \subset S$ is a Gilmer extension. In that case,   each minimal subextension is inert, with separable minimal residual field extensions.
 \end{proposition}
 
 \begin{proof} Assume that $R/(R:S)$ is Artinian. We can suppose that $(R:S)= 0$ and that  $(R,M)$ is local Artinian with $M\in \mathrm{Supp}(S/R)= \mathrm{Spec}(R)$  and then $S$ is Artinian. As $R\subset S$ is $t$-closed,  there is only one prime ideal $N$ of $S$,  lying over $M$ \cite[Th\'eor\`eme 3.11]{Pic 1} and then $(S,N)$ is local.
 Property UN implies that  $MS= N$ and that  the residual extensions are separable, whence have FIP. Since $R\subset S$ is seminormal,  $(R:S)$ is semiprime, so that $(R:S)= N= M$. Since $R/M \subset S/N$ is separable, this extension has FIP and so does $R\subset S$. It follows that any decomposition of $R\subset S$ into minimal extensions is composed of inert extensions, whose residual  field extensions are separable and minimal. The converse holds because the property FIP implies the property FCP and then we can use Proposition~\ref{SUP}(3).
 \end{proof}

We intend to study \'etale Gilmer extensions.  

Let $R\subseteq S$ be a finite extension and $P\in \mathrm{Spec}(R)$. In the next proposition, we set  $n(P):= |\mathrm{Spec}(\kappa (P)\otimes_RS)|$, that is the cardinal of the fiber at  $P$ of $R\subseteq S$ and by $Q_1,\ldots ,Q_{n(P)}$ the elements of this fiber.

\begin{theorem}\label{2.4} Let $R$ be an Artinian  ring and $R \subseteq S$  an \'etale extension.  The following statements hold:
\begin{enumerate}

\item $R\subseteq S$ is a projective finite extension.

\item Let  $Q \in \mathrm{Spec}(S)$ and $P:= Q\cap R$.  Then  $R_ P\to S_Q$ is an    extension and $R_P\subseteq S_P$ is free, with finite rank. Moreover, $PS$ is a radical ideal of $S$, equal to $Q_1\cap \cdots \cap Q_{n(P)}$.

\item  There is an \'etale   extension $R_P^{n(P)} \subseteq S_P$ where $n(P) \leq [S/PS:R/P]= \sum_{i= 1,\dots ,n(P)} [\kappa (Q_i) : \kappa(P)]$. 

\item  $R\subseteq S$ is infra-integral if and only if  $n(P) = [S/PS:R/P]$ for each $P\in \mathrm{Spec}(R)$, and if and only if $S_P$ is isomorphic to $(R_P)^{n(P)}$ for each $P\in \mathrm{Spec}(R)$.

\item  $R\subseteq S$ is subintegral  if and only if $R=S$

\item   $R_P\to S_Q$ is \'etale, finite and has FIP if $R\subseteq S$ has FIP.

\end{enumerate}
\end{theorem}
\begin{proof} 

(1) $R\subseteq S$ is projective finite by Proposition~\ref{GUF}.

(2) Observe that the maps $R_P \to S_P$ and $R_P\to S_Q$ are extensions because they are faithfully flat.
Actually, $R_P \to S_P$ is free, because flat and finite.   Moreover, $PS$ is a radical ideal by Proposition~\ref{RED}.
 
 (3) Since  $S_P$ is Artinian, it is isomorphic to $\prod [S_{Q_i} | Q_i \in \mathrm{Spec}(S_P)]$, where $Q_1,\ldots, Q_n$ is the  finite set of maximal ideals of $S$ lying over $P$. There is a factorization $R_P \to (R_P)^n \to S_P$. Since $R_P\to (R_P)^n$ is \'etale so is $(R_P)^n \to S_P$ by Lemma~\ref{Seppro}.   Then we get $n \leq [S/PS:R/P]$, because $\kappa (P)^n \subseteq \kappa(P)\otimes_RS$ and $\kappa(P) = R/P$.
 
 (4)    Set $n:= n(P)$ for a prime ideal $P$ of $R$. If $n= [S/PS: R/P$], we get by (3) that $\kappa (P) \cong \kappa (Q_i)$ for each $i$ and $R_P\to S_P$ is infra-integral. The converse is clear. Assume that $R_P\subseteq S_P$ is infra-integral, then $ R_P^n/PR_P^n \to S_P/PS_P$ is an isomorphism, because by (2) we have $S/PS\cong \prod_{i=1}^n S/Q_i \cong (R/P)^n $, so that $S_P/PS_P \cong (R/P)^n_P \cong (R_P/PR_P)^n$.
  
 The Nakayama Lemma implies that $R_P^n = S_P$ since $PR_P$ is a nilpotent ideal \cite[Corollaire 1, p.105]{Bki AC}. For the converse,  use Lemma~\ref{2.1}. 
 
 (5) Suppose that in addition to the hypotheses of (4),  $\mathrm{Spec}(S) \to \mathrm{Spec}(R)$ is injective. It follows from (4), that each $n(P) = 1$ and $R_P = S_P$.
 
(6)  Now $R_P \to S_Q$ is \'etale because $R_P\to S_P$ is \'etale and so is  $S_P \to S_Q$. Indeed,   $S_P\to S_Q$ is a flat epimorphism, which is surjective by the  Property (SE) and is therefore of finite presentation, because $S_P$ is Noetherian. Suppose that in addition $R\subseteq S$ has FIP. We have just shown that $R_P \to S_Q$ is of the form $R_P \to S_P \to S_Q$, where $S_P \to S_Q$ is surjective.  Denote by $J$ the ideal of $ S_P$ such that $S_Q=S_P/J$. We have $J\cap R_P=0$. We are in position to apply Proposition~\ref{QUOT} since $R_P\subseteq S_P$ has FIP and  then $R_P \to S_Q$ has FIP.
 \end{proof}
 
 Hence if $R$ is an Artinian ring and $R\subseteq S$ is \'etale and has FIP, then $R_P\subseteq S_Q$ has FIP. This means that $R\subseteq S$ has locally FIP. In view of  Lemma~\ref{2.1} the same result is valid for the \'etale extension $R\subseteq R^n$, even if $R$ is Artinian and $R\subseteq R^n$ has not FIP (see Theorem~\ref{PROD}). Thus an extension that has  locally FIP  does not need to have FIP.

\begin{corollary}\label{2.5} Let $R\subseteq S$ be an \'etale Gilmer  extension.  The following statements hold:

\begin{enumerate} 

\item  $R\subseteq S$ is finite and projective and $\mathrm{V}_R((R:S)) =\mathrm{Supp}(S/R)$ is a finite subset of $ \mathrm{Max}(R)$;

\item  $R_P\subseteq S_Q$ is \' etale  and finite for  $Q\in \mathrm{Spec}(S)$ and  $P:= Q\cap R$.  Moreover, $R_P\subseteq S_Q$  has FIP if $R\subseteq S$ has FIP;

\item  $R\subseteq S$ is infra-integral if and only if  $S_M$ is $R_M$-isomorphic to some $(R_M)^n$ for each $M\in \mathrm{Supp}(S/R)$ (where necessarily $n = n(M))$;

\item $R= S$ if and only if $R\subseteq S$ is subintegral.
\end{enumerate}
\end{corollary}
\begin{proof}
 Thanks to Proposition~\ref{GUF}, we get that $R \subseteq S$ is finite  so that $\mathrm{Supp}(S/R) = \mathrm{V}((R:S))$. It is projective by Proposition~\ref{GUF}.  For the statements about infra-integrality and subintegrality, use Proposition~\ref{PROPN} and Theorem~\ref{2.4}.
\end{proof}
We thus find that an \'etale finite extension is never  subintegral, an observation that already appeared in Remark~\ref{NOUN}. 

\begin{theorem}\label{2.51} Let $R\subseteq S$ be an infra-integral Gilmer  extension. Consider the following statements:
\begin{enumerate}
\item  $R\subseteq S$ is  \'etale. 

\item  $R\subseteq S$ is of finite presentation and is locally a trivial \'etale cover; that is, for each $P\in \mathrm{Spec}(R)$,  $S_P\cong (R_P)^{n(P)}$ for some integer $n(P)$.

\item $R\subseteq S$ is locally a trivial \'etale cover and $R$ is an $\mathcal S$-ring (for instance is either semi-local or whose total quotient ring is semi-local).
\end{enumerate}

Then (1) $\Leftrightarrow$ (2) and (3) $\Rightarrow$ (1). 
\end{theorem}

\begin {proof} (1) $\Leftrightarrow$ (2). Assume that $R\subseteq S$ is  \'etale. Then, locally $S\cong R^n$ for some integer $n$ in view of Corollary~\ref{2.5}.  Conversely, since $R\subseteq R^n$ is  \'etale for any integer $n$ by Proposition~\ref{1.5},  $R\subseteq S$ is  flat,  unramified and of finite presentation.

(3) $\Rightarrow$ (1) because $R\subseteq S$ is then flat and unramified, whence finite. To conclude, $R\subseteq S$ is projective by definition of an $\mathcal S$-ring (see Remark~\ref{LOCAL}) and is then of finite presentation by Lemma~\ref{Seppro}, so that (2) is verified.
\end {proof}

\begin{theorem}\label{2.6} Let $R\subseteq S$ be an \'etale infra-integral Gilmer extension. Then $R\subseteq S$ has FIP if and only if the nucleus of $R\subseteq S$ is a FMIR and the cardinal of each fiber  at a prime ideal $M\in \mathrm{Supp}(S/R)$ is $2$ as soon as $R_M$ is an infinite SPIR.
\end{theorem}

\begin {proof} We saw that under the hypotheses $(R:S)_{(R:S)} =0$, so that $(R:S)_M=0$   for $M\in \mathrm{Supp}(S/R)$. Set $R':= R_{(R:S)}$ and $M':= M_{(R:S)}$, for $M\in \mathrm{Supp}(S/R)$. We get that $R_M \cong (R')_{M'}$. To conclude, it is enough to apply  Theorem~\ref{PROD} and Corollary~\ref{2.5}.
\end{proof}

It is worth noticing that Theorem~\ref{2.6} gives an answer for the part $R\subseteq \substack{t\\ S}R$
of the canonical decomposition of $R\subseteq S$.
The next result  has a simple proof because it concerns  an extension whose domain is Artinian reduced, whence absolutely flat.

\begin{theorem}\label{2.7} Let $R \subset S$ be an unramified Gilmer extension, whose nucleus is reduced. Then $R \subset S$ is finite and has FIP.
\end{theorem}
\begin{proof} There is no harm to  assume  that $R$ is an Artinian ring with maximal ideals $M_1,\ldots ,M_n$, where $n>0$ is an integer. Proposition~\ref{GUF} shows that $R\subset S$ is finite.  Moreover,  $\mathrm{Supp}(S/R)$ is finite since so is the spectrum of $R$. For each $i$, $R_{M_i} \subset S_{M_i}$ is \'etale over the field $R_{M_i}$. In view of Theorem~\ref{ET}, this extension has FIP. Since the support of $S/R$ is finite, we get that $R\subset S$ has FIP by Proposition~\ref{SUP}(2).  
\end{proof}

\begin{remark}\label{2.8}(1) We now exhibit an \'etale  finite extension $R\subseteq R^n:= S$ with conductor $0$ by Lemma~\ref{2.1}, for any integer $n> 2$, which has not FIP, showing that supposing only $R$ with finitely many ideals is not sufficient, and  such that $R$ is local Artinian and is not reduced. Let $K$ be an infinite field and set $R:= K[X]/(X^2)$ where $X$ is an indeterminate over $K$. 

 Then $R$ is a local Artinian ring, in fact, an infinite SPIR, with maximal ideal $M$, such that $M^2=0$ and $R/M \cong K$. 
 Then, $R\subset R^n$ is \'etale by Proposition~\ref{1.5}, but has not FIP by Theorem~\ref{PROD}.

(2) Consider the faithfully flat \'etale finite extension
$R\subseteq R^2=: S$ where we choose a   ring $R$ with finitely many ideals, for instance a finite ring. We can write $S = R + Rt$ for some $t\in S$. In view of \cite[Proposition 4.12]{Gil},  there is a bijection from $[R,S]$  to the set of ideals of $R$.  It follows that $R\subseteq S$ has FIP whereas $R$ is not  necessarily reduced. There are no rank conditions because they are hidden by the choice of  $R\subseteq R^n$, with $n=2$.

(3) Here is an example of an unramified infra-integral FIP extension which is not seminormal. Let $(R,P)$ be a SPIR such that $P=Rz$, with $z^2=0$ (take $R:=k[Z]/(Z^2)=k[z]$, where $k$ is a field).  Set $T:=R[X]/(zX,X^2)=R[x]$ and $S:=T[Y]/(Y^2-Y,zY-x)=T[y]$. It is easy to see that $R\subset T$ is minimal ramified, and $T\subset S$ is minimal decomposed (\cite[Theorem 2.3]{DPP2}). Then, $R\subset S$ has FIP, is infra-integral and not seminormal. Moreover, $M:=P+Rx$ is the only prime ideal of $T$, and $Q:=M+Ty,\ Q'=M+T(1-y)$ are the two prime ideals of $S$. We get that $PS=M=QQ'$, giving $PS_Q=QQ'S_Q=QS_Q$ and $PS_{Q'}=QQ'S_{Q'}=Q'S_{Q'}$, so that $R\subset S$ is unramified, since the residual extensions are isomorphisms.

(4) Another example is as follows and shows that under the hypothesis of Theorem~\ref{2.4} the rank of $S$ over $R$ may not be defined. Let $R$ be an Artinian reduced ring and  $f_1,\ldots, f_n \in R$ such that $(f_1,\ldots,f_n) = R$; we set $S:= \prod_{i=1}^n R_{f_i}$. It is known that $R\to S$ is a faithfully flat \'etale morphism. Localizing the extension at each $P\in\mathrm{Spec}(R)$, we get morphisms of the form $R_ P\to S_P\cong (R_P)^p$ where $p$ is the number of open subsets $\mathrm{D}(f_i)$ containing $P$ while  $R_P$ is a field. In view of Proposition~\ref{SUP}(2) and Theorem~\ref{ET}, $R \to S$ has FIP.

\end{remark}

\subsection{The case of a base ring with finitely many ideals}

 We intend to generalize Theorem~\ref{PROD} to \'etale extensions $R\subseteq S$ that may not be infra-integral as in a trivial \'etale cover.
 
\begin{theorem}\label{2.9} Let $R \subseteq S$ be an \'etale extension, where $R$ is a FMIR. 
\begin{enumerate}

\item The ring $S$ is a FMIR  and if $Q\in \mathrm{Spec}(S)$ and $P:= Q\cap R$, then $S_Q$ is a field (respectively, a finite ring, an infinite SPIR) if and only if $R_P$ is a field (respectively, a finite ring, an infinite SPIR).
\item If in addition, $\mathrm{rk}_P(S_P)\leq 2$ for each prime ideal $P\in \mathrm{Supp}_R(S/R)$ such that $R_P$ is an infinite SPIR, then $R\subseteq S$ has FIP.

\end{enumerate}
\end{theorem}
\begin{proof} We begin with the proof of (2).  We can assume that the conductor of $R\subseteq S$ is zero and we will prove that $R_P \subseteq S_P$ has FIP for each $P\in \mathrm{Supp}(S/R)$. If $R_P$ is either  a field or a finite ring, we know that $R_P \subseteq S_P$ has FIP, because $R\subseteq S$ is finite. The only case to consider is when $A:= R_P$ is an infinite SPIR with maximal ideal $M$ and then  we have $\mathrm{rk}_P(S_P) =2$. In this case $A\subseteq S_P$ is an \'etale cover, because this extension is finite.  Denote by $\widehat A$ a strict Henselization of $A$. The ring $\widehat A$ is local with maximal ideal $\widehat M= M\widehat A$ and $A \to \widehat A$ is a faithfully flat ring morphism \cite[pp.94-95]{R}. It follows that $\widehat A$ is zero-dimensional and Noetherian, whence Artinian \cite[pp.94-95]{R}. Moreover, if $M= At$ and $M$ is nilpotent of index $p> 0$, we have $\widehat M= \widehat At$ and $\widehat M$ is nilpotent of index $p> 0$.  It follows that $\widehat A$ is an infinite SPIR. We deduce from \cite[Proposition 18.8.1]{EGAIV} stating that an \'etale cover over a strict Henselian ring is trivial, that $S_P\otimes_A \widehat A\cong (\widehat A)^n$, for some integer $n$, necessarily equal to $2$. Therefore, $\widehat A \subseteq S_P\otimes \widehat A$ has FIP by Theorem~\ref{PROD}. Since $A \to \widehat A$ is faithfully flat, it descends morphisms that have FIP \cite[Theorem 2.2]{DPP3}. We have thus proved that $R\subseteq S$ has FIP.

(1) Now let $Q$ be a prime ideal of $S$ lying over $P$ in $R$. Since $S$ is Artinian, we can identify this ring with a finite product of its localizations at prime ideals. Therefore, to prove that $S$ has finitely many ideals, it is enough to prove that  the rings $S_Q$ have the same property. Now since $S_Q$ is a localization of $S_P$, we need only to prove that $S_P$ has only finitely many ideals.  This is clear if $R_P$ is either a field or a finite ring. In case $A:= R_P$ is an infinite SPIR, the above argumentation without rank hypotheses  shows that $S_P \to( \widehat A)^n$ is faithfully flat and clearly $(\widehat A)^n$ has finitely many ideals. To complete the proof it is enough to observe that for a faithfully flat ring morphism $B\to C$ and $I$ an ideal of $B$, we have $IC\cap B= I$.

Now by Theorem~\ref{2.4}(6), $R_P\subseteq S_Q$ is faithfully flat, finite and \'etale. If $R_P$ is a field, then $S_Q$ is a field, since isomorphic to a product of finitely many fields, because $R_¬\subseteq S_Q$ is \'etale. If $R_P$ is finite, so is $S_Q$. If $R_P$ is an infinite SPIR, then $S_Q$ cannot be neither a field nor finite. Deny,  then $R_P$ would be either a domain or finite. If $S_Q$ is either finite or a field, so is $R_P$. Suppose that $S_Q$ is an infinite SPIR.  If $R_P$ is finite, then $\kappa (P)$ is finite and $\kappa (P) \subseteq \kappa (Q)$ is finite so that $\kappa (Q)$ is finite. This in turn implies that $S_Q$ is finite, an absurdity. Now if $R_P$ is an infinite field, then $S_Q$ is a product of fields, because $R_P\subseteq S_Q$ is \'etale, an absurdity. It follows that $R_P$ is an infinite SPIR. 
\end{proof}

\begin{corollary}\label{2.10} Let $R\subseteq S$ be an \'etale Gilmer extension, whose nucleus is a FMIR. If  $\mathrm{rk}_{R_P}(S_P) \leq 2$ for each $P\in \mathrm{Supp}_R(S/R)$ such that $R_P$ is an infinite SPIR, then $R \subseteq S$ has FIP. 
\end{corollary}
\begin{proof}  We can replace $R$ with $R_{(R:S)}$, so that the base ring is a FMIR, because $(R_{(R:S)}: S_{(R:S)})= 0$. Then we can  use  Theorem~\ref{2.9}, because  the spectrum of $R':= R_{(R:S)}$ is the set of all extensions $M'$ of prime ideals $M$ in $\mathrm{Supp}(S/R)$, so that $R_M\to R'_{M'}$ is an isomorphism.\end{proof}

\begin{corollary}\label{2.10} Let $R$ be a ring which is a product of finitely many fields and finite rings and $R\subseteq S$, $S\subseteq T$ two \'etale extensions. Then $S\subseteq T$ has FIP.
\end{corollary}

\begin{remark}\label{2.11}  Let $R$ be a ring with finitely many ideals and $p(X) \in R[X]$ a monic polynomial. We set $R_1:= R[X]/(p(X))$. If the degree of $p(X)$ is two, then $R\subseteq R_1$ has FIP by  \cite[Proposition 4.12]{Gil}, which states that there is a bijection from $[R,S]$ to the set of ideals of $R$.  So from now on, we suppose that $n>2$  and that $p(X)$ is  a separable polynomial, that is, $p(X) a(X) + p'(X)b(X) = R[X]$ for some $a(X), b(X) \in R[X]$.  We have $p(X) = (X-a_1)p_1(X)$ where  $a_1$ is the class of $X$ in $ R_1$ and  some $p_1(X) \in R_1[X]$. Then $R\subseteq R_1$ is faithfully flat, finite, \'etale and its conductor is $0$. Note that we can apply Theorem~\ref{2.8} if there is no infinite SPIR in the product of rings defining $R$, because the rank of $R_1$ over $R$ is constant and $>2$.

\end{remark}

\section{The Local  Case} We intend to  examine subintegral extensions when the base ring is local Henselian (for example Artinian). We recall that this kind of extension is never unramified.
\subsection{The FIP property via local subextensions}
We begin with a generalization of a a result of M. Kosters, established for a local Artinian ring \cite[Lemma 4.12]{K}.
It will allow us to reduce the study of the FIP property of a finite extension $R\subseteq A$, where $(R,M)$ is a local Henselian ring to local subextensions $R\subseteq B$ of $R\subseteq A$, where $(B,N)$ is a local ring. We observe that $R\subseteq B$ is a local morphism. We denote by $[R,A]_{loc}$ the set of all subextensions of $R\subseteq A$ that are local. 

We set $A_X:= \prod \, [A_N\, | \, N\in X]$, for any $X\subseteq \mathrm{Max}(A)$ and denote by $\Pi_A$ the set of all partitions of the finite set $\mathrm{Max}(A)$.
If $\{S_1,\ldots,S_n\}$ is a partition of $\mathrm{Max}(A)$, we set $A_i:= A_{S_i}$, for $i=1,\ldots,n$.

\begin{proposition}\label{4.1} Let $(R,M)$ be a local Henselian ring and $R\subseteq A$ a finite ring extension, so that $A\cong A_{\mathrm{Max}(A)}$.

  \item[\, \,(1)] Let $\{S_1,\ldots,S_n\}$ be a partition of $\mathrm{Max}(A)$ and let $B_i \in [R,A_i]_{loc}$  with maximal ideal $M_i$ for each $i\in \{1,\ldots,n\}$. Then $B:= \prod_{i=1}^n B_i $ belongs to $[R,A]$, $\mathrm{Max}(B) = \{ N_1,\ldots, N_n\}$, where $N_i:= M_i\times \prod_{j\neq i} B_j$. Moreover, the fiber of $N_i$ in $A$ is  $S_i$ and $B_{N_i} \cong B_i$.

\item [\, \,(2)] There is an injective map $\varphi: [R,A] \to \Pi_A\times  \{ \bigcup_{\mathcal P \in \Pi_A}\prod_{X\in \mathcal P} [R,A_X]_{loc}\}$. Moreover, the image of  $ \varphi$ is  $ \{(\mathcal P, \prod_{X\in \mathcal P}\, [R,A_X]_{loc} \, | \mathcal P \in \Pi_A\}$.
\end{proposition}
\begin{proof} By the very definition of an Henselian ring, we have that $A\cong  A_{\mathrm{Max}(A)}$ \cite[Chapter 1]{R}. Hence we will replace $A$ by the product.

(1) Since we  have $B_i\subseteq \prod \, [A_N \, | N\in S_i]$, it is clear that $B\subseteq A$, because $\{S_i\}_{i=1}^n$ is a partition of $\mathrm{Max}(A)$.  Then $\mathrm{Max}(B)$ is well known, within the ordering. It follows from Remark~\ref{LOCPROD} that  $B_{N_i}\cong B_i$. The maximal ideals of $A_i$ are lying over $M_i$. Observe  that there is a bijection $S_i \to \mathrm{Max}(A_i)$. We deduce from this fact that each element of $S_i$ is lying over $N_i$. Now if $P$ is a maximal ideal of $A$ which does not belong to $S_i$, there is some $j\neq i$, such that $P\in S_j$, and then $P$ contracts to $N_j\neq N_i$. Hence the fiber of $N_i$ in $A$ is $S_i$. 

(2)  Let $B\in [R,A]$, the map $\mathrm{Max}(A) \to \mathrm{Max}(B)$, defined by $P\mapsto P\cap B$,   defines an equivalence relation whose classes are a partition $\mathcal P =\{S_1,\ldots, S_n\}$ of $\mathrm{Max}(A)$. Set $N_i := P\cap B$ for (each) $P\in S_i$. 

 Now \cite[Proposition 2, p. 7]{R} shows that $B_{N_i}$ is integral over $R$ and is Henselian. Therefore, $A_{N_i} $ is  the product of its local rings and finally we get that $A_{N_i} \cong \prod \,[A_P \, | \, P\cap B= N_i]= A_i$. We deduce from (1) that $B_{N_i} \in [R,A_i]_{loc}$.
We then set $\varphi (B) := (\mathcal P, (B_{N_1},\ldots, B_{N_n}))$, where each $B_{N_i}\in [R,A_i]_{loc}$.  Suppose that $\varphi(B) = \varphi(C)$ for $B, C \in [R,A]$. Set $B_i:=B_{N_i}$ and $B'=\prod_{i=1}^nB_i$. Then, $\mathrm{Max}(B')=\{N'_1,\ldots,N'_n\}$, where $N'_i:=N_i\times \prod_{j\neq i}B_j$. By (1), we have $B'_{N'_i}=B_i=B_{N_i}$. But, $B\subseteq B'$ and $N'_i$ is the only maximal ideal of $B'$ lying over $N_i$. Then, $B'_{N'_i}=B'_{N_i}=B_{N_i}$, so that $B=B'=\prod_{i=1}^nB_i=\prod_{i=1}^nC_i=C$.
 
It follows that $\varphi$ is injective.
\end{proof}
We infer from the above proposition the next result. 

\begin{theorem}\label{4.2} Let  $R\subseteq A$ a finite ring extension, where $R$ is a local ring.  The following statements are equivalent:

\item[\, (1)] $R\subseteq A$ has FIP;
\item[\, (2)]  $[R,A]_{loc}$ is finite and $R/(R:A)$ is Henselian;
\item[\, (3 )]  $[R,A]_{loc}$ is finite and $R/(R:A)$ is 
Artinian.
\end{theorem}

We note here a generalization of a result of M. Kosters, given in a more general context but with less results, because our context is richer.

\begin{proposition}\label{4.3} (M. Kosters) Let $(R,M)$ be an infinite local ring and $R\subseteq S$ a Gilmer extension. Then $R\subseteq S$ has FIP if and only if $R/M\subseteq S/MS$ has FIP, $MS/M$ is a uniserial $R$-module and $R \subseteq R + \mathrm{Nil}(S)$ has FIP.
\end{proposition}
\begin{proof} To prove this, first observe that we can reduce to the case $(R:S)=0$. Then  use some parts of the proof of \cite[Proposition 4.13]{K}, the step II being superfluous because of Theorem~\ref{4.2}.
\end{proof}

Proposition~\ref{4.3} is  trivial if either $(R:S) = M$ or $R\subseteq S$ is infra-integral, because in that case each $T\in [R,S]$ is contained in $R+ \mathrm{Nil}(S)$.  We will examine the condition $MS/M$ is uniserial in the following.

\begin{proposition}\label{4.4} Let $(R,M)$ be an infinite local  ring and an unramified Gilmer extension $R\subseteq S$, whose nucleus is  not a  field. Then $R\subseteq S$ has FIP if and only if $R \subseteq R+MS$ has FIP and $MS/M$ is a uniserial  $R$-module. In case $(S,N)$ is local, then $MS/M= N/M$.
\end{proposition}
\begin{proof} We know that if $R\subseteq S$ is unramified, then $R/M \subseteq S/MS$ is \'etale and then $\mathrm{Nil}(S) = MS$, because $S/MS$ does
be reduced. Moreover, an \'etale algebra over a field has FIP. For the last statement use Property~UN.
\end{proof}

\subsection{The subintegral part of an \'etale extension}

Let $R$ be ring and $n>1$ a positive integer. We have seen in Lemma~\ref{2.1} that $R\subseteq R^n$ is an infra-integral \'etale extension, which is not seminormal if $R$ is not reduced. In this case, there is a subintegral part $R\subset \substack{+\\ R^n}R$ of the extension $R\subset R^n$. We consider the inverse problem. Given a subintegral extension $R\subset T$, does there exist some infra-integral \'etale extension $R\subset S$ such that $T\in[R,S]$? Assuming that $R$ is a local Artinian ring, Theorem~\ref{2.51} can be used. We first give a general result.

\begin{proposition}\label{4.40} Let $R\subset T$ be a ring extension where $R$ is a local Artinian ring. There exists an infra-integral \'etale extension $R\subset S$ such that $T\in[R,S]$ if and only if there exists some integer $n$ such that $T\in[R,R^n]$ (up to an isomorphism of $R$-algebras). 
\end{proposition}
\begin{proof} Assume that there exists some integer $n$ such that $T\in[R,R^n]$ (up to an isomorphism $\varphi$ of $R$-algebras). Set $S:=\varphi(R^n)$. Since $R\subset R^n$ is an infra-integral \'etale extension by Lemma~\ref{2.1}, so is $R\subset S$ and we have $T\in[R,S]$. Conversely, assume that there exists an infra-integral \'etale extension $R\subset S$ such that $T\in[R,S]$. Then, $R\subset S$ is a Gilmer extension since $R$ is Artinian. Applying Theorem~\ref{2.51}, we get that $S\cong R^n$ for some integer $n$, giving that $T\in[R,R^n]$ (up to an isomorphism of $R$-algebras).  
\end{proof} 

The following example shows how to build an extension $R\subset S$ satisfying the conditions of Proposition~\ref{4.40}, given a particular minimal ramified  (and then subintegral) extension $R\subset T$.  

\begin{example}\label{4.41} Let $(R,M)$ be a SPIR, with $M=Rt$, satisfying $t^2=0,\ t\neq 0$ and let $R\subset T$ be a minimal ramified extension such that $T=R+Rx$, with $tx=x^2=0$ (it is enough to take $T:=R[X]/(X^2,tX)$). Then, $P:=Rx+Rt$ is the maximal ideal of $T$ and $M=(R:T)=(R:x)$. There is an  injective morphism of $R$-agebras  $\psi: T\to R^2$ such that $\psi(a+bx)=(a+bt,a)$. Indeed, $\psi$ is well defined. Let $z\in T$, with $z=a+bx=c+dx,\ a,b,c,d\in R$. Since $(d-b)x=a-c\in R$, it follows that $d-b\in M$, so that there exists $\lambda\in R$ such that $d-b=\lambda t$. Then, we get that $a-c=\lambda tx=0$, giving $a=c$ and $b=d-\lambda t$, so that $a+bt=c+(d-\lambda t)t=c+dt$. The same reasoning shows that $\psi$ is injective, since $\psi(a+bx)=(a+bt,a)=(0,0)$ implies first, that $a=0$, and then $bt=0$, so that $b\in M$, giving $bx=0$.  Hence, $\psi(T)$ is an $R$-subalgebra of $R^2$. Identifying $T$ and $\psi(T)$, so that $1=(1,1)$ and $x=(t,0)$, we get that $T\subset R^2$ is a minimal decomposed extension, with $R^2=T[y]$, where $y=(1,0)$ satisfies $y^2-y=0\in P$ and $Py=Rxy+Rty=Rx\subseteq P$. It follows that $T=\substack{+\\ R^2}R$, so that $T\in [R,R^2]$, where $R\subset R^2$ is an infra-integral \'etale extension.
\end{example}

According to Proposition~\ref{4.40}, we are led to the following problem. Given a subintegral FIP extension $R\subset T$, what are the conditions in order that $T\in[R,R^n]$ for some integer $n$? We first observe that $R$ must be non reduced, since  $R\subset R^n$ is not seminormal (Lemma~\ref{2.1}). In order that $R\subset R^n$ has FIP,  $R$ is necessarily an FMIR (Theorem~\ref{PROD}), and Proposition~\ref{4.40} forces $R$ to be either a non-reduced finite local ring  or an infinite SPIR (in which case $n=2$). 

\begin{proposition}\label{4.42} Let $R\subset T$ be a subintegral FIP extension where $(R,M)$ is either a finite local ring which is not a field, or a SPIR. 
\begin{enumerate}
\item There exists  some integer $n$ (with $n=2$ if $R$ is an infinite SPIR) such that $T=\substack{+\\ R^n}R$ if and only if $T=R+MR^n$.

\item There exists  some integer $n$ (with $n=2$ if $R$ is an infinite SPIR)  such that $T\in[R,R^n]$  if and only if $T=R+N$ for some $R$-submodule $N$ of $MR^n$ containing $M$ satisfying $N^2\subseteq N$. 
\end{enumerate}
\end{proposition}
\begin{proof} Since $R\subset T$ is a subintegral FIP extension, it follows that $T$ is a local Artinian ring. Let $N$ be its maximal ideal, so that $T=R+N$ because $R/M\cong T/N$. Moreover, $R$ is not reduced.

(1) Assume that $T=\substack{+\\ R^n}R$. Then, $T\subset R^n$ is a seminormal extension with $(T: R^n)=N=MR^n$ in view of Proposition~\ref{2.2}, since $N$ is the only prime ideal of $T$ and $MR^n$ is the intersection of the prime (maximal) ideals of $R^n$. So, we get that $T=R+MR^n$. 

Conversely, assume that $T=R+MR^n$. Since $(T: R^n)=MR^n$ is a radical ideal of $R^n$,  then $T\subset R^n$ is a seminormal extension by Proposition~\ref{2.2}, so that $T=\substack{+\\ R^n}R$ since $R\subset T$ is  subintegral.

(2) Assume that there exists some integer $n$ such that $T\in[R,R^n]$. Then, $T\subseteq \substack{+\\ R^n}R$ since $R\subset T$ is  subintegral. It follows that $M\subseteq N\subseteq MR^n$ in view of (1), with $N$ an $R$-module. But, as $N$ is an  ideal of $T$, we have $N^2\subseteq N$.

Conversely, assume that $T=R+N'$ for some $R$-submodule $N'$ of $MR^n$ containing $M$ satisfying $N'^2\subseteq N'$.  Then, $T\subset R^n$, giving $T\in[R,R^n]$. Moreover, $N'^2\subset N'$ implies that $TN'=RN'+N'^2=N'$. Hence, $N'$ is an ideal of $T$ such that $N'\cap R=M$. Moreover, $N'$ is a prime ideal by a Noether's isomorphism for rings $R/M \cong T/N'$,  so that $N'=N$.  
\end{proof}

\subsection{Subintegral extensions over local rings}

We now give some  results about (necessarily local) subintegral extensions $R\subseteq S$ having  FIP in case they are Gilmer extensions. We need a result about an extension $R\subseteq S$, where $(R/(R:S),M/(R:S))$ is a local Artinian ring. We will set $R_k: = R+ M^{n-k}S$, giving rise to a sequence $R_0:= R \subseteq R_1\subseteq \cdots\subseteq R_{n-1}= R+MS  \subseteq S$, where $n \geq 1$ is the nilpotency index of $M/(R:S$). It is clear that each extension $R_k \subseteq R_{k+1}$ is a $\Delta_0$-extension in the sense of \cite{HP}, that is each $R_k$-submodule of $R_{k+1}$, containing $R_k$, is an element of $[R_k,R_{k+1}]$. It follows that each element of $R_{k+1}$ is a zero of a monic polynomial of $R_k[X]$, with degree $2$ ({\it  loc.cit.}). We will need some lemmata. In view of \cite[Proposition 3.7(c)]{DPP2},  there is a preserving order bijection $[R, S]\to[R/(R:S) ,S/(R:S)]$ given by $T\mapsto T/(R:S)$.

\begin{proposition}\label{4.5} Let $(R,M)$ be a local ring such that $R/M$ is infinite,  and let $R\subset S$ be a finite, subintegral Gilmer extension.
 The following statements are equivalent:
\begin{enumerate}
\item $R\subset S$ has FIP;
\item $R\subset S$ is chained;
\item If $n>1$, then $[R,R_{n-1}] =  \{R_k \, | \, k=0,\ldots , n-1 \}$ and $|[R_{n-2},S]|\leq 4$. If $n=1$, then $|[R,S]|\leq 4$;
\end{enumerate}

If one of the previous statements holds, then $|[R,R_{n-1}]| = n$ and, when $n>1$,  each $R_k\subset R_{k+1}$ is a minimal ramified extension, for $k=0, \ldots, n-2$. Moreover, there exists $\alpha\in S$ such that $S=R_{n-1}[\alpha]$ with $\alpha^3\in MS$, and one of the three following situations is satisfied: either $R_{n-1}=S$, or $R_{n-1}\subset S$ is a minimal ramified extension, or there exists $T\in[R_{n-1},S]$ such that both $R_{n-1}\subset T$ and $T\subset S$ are minimal ramified extensions. 
\end{proposition}
\begin{proof} We know that $R\subset S$ has FCP by \cite[Theorem 4.2]{DPP2}. 
Then (1) $\Rightarrow$ (2) by \cite[Proposition 4.13]{DPP3}. Indeed, this Proposition says that $|[(R,S)]|=\ell[R,S]+1$, where $\ell[R,S]$ is the supremum of the lengths of any maximal chain of $R$-subalgebras of $S$. Now,  (2) $\Rightarrow$ (1) because $[R,S]$ is chained, so that $[R,S]$ is composed of the elements of the unique finite maximal chain going from $R$ to $S$. 

For the equivalence (1) $\Leftrightarrow$ (3), we may assume that $(R:S)=0$, in view of a preceding remark. We first assume that $M\neq 0$, so that $R$ is not a field. In view of \cite[Proposition 5.15]{DPP2}, (1) is equivalent to $R\subseteq R_{n-1}$ and $R_{n-2}\subseteq S$ have FIP. But $R\subseteq R_{n-1}$ has FIP if and only if $[R,R_{n-1}] =  \{R_k \, | \, k=0,\ldots , n-1 \}$ by \cite[Lemma 5.12]{DPP2} and $R_{n-2}\subseteq S$ has FIP if and only if $|[R_{n-2},S]|\leq 4$ by \cite[Lemma 5.14]{DPP2}, so that (1) $\Leftrightarrow$ (3).

Assume now that $M=0$, so that $R$ is a field. Then, $n=1$ and (1)   $\Leftrightarrow |[R,S]|\leq 4\Leftrightarrow$ (3) by the proof of \cite[Lemma 3.6 (b)]{ADM}. 

We come back to the general case.

If one of the previous statements holds, then $|[R,R_{n-1}]| = n$ and, when $n>1$, each $R_k\subset R_{k+1}$ is a minimal ramified extension, for $k=0, \ldots, n-2$ by the proof of (3). Moreover, \cite[Lemma 5.13]{DPP2} gives the last part of the statement, first in $S/(R:S)$, and then in $S$.
 \end{proof}

 \begin{lemma}\label{4.6}  Let $(R,M)$ be a   local ring such that $R/M$ is infinite and $R\subset S$  a finite subintegral Gilmer extension, whose nucleus is  not a field and has  $n$ as nilpotency  index. Let $M_i := M + M^{n-i}S$, then   if $R\subset R+MS$ has FIP, there exists some $x_i \in M_i \setminus M_{i-1}$ such that $ M_i = M+ Rx_i$,  for $i= 1,\ldots , n-1$.
\end{lemma} 
\begin{proof}  We can apply Theorem~\ref{CANMIN} to $R\subset R+MS$ which states that $R_{i-1}\subset R_i$  is a minimal ramified extension, so that  the length of the $R_{i-1}$-module $ M_i/M_{i-1}$ is $1$  \cite[Theorem 3.3]{Pic}. We deduce from \cite[Theorem 13, p.168]{N}, by the Northcott   extension formula for length that the length over $R$ of $ M_i/M_{i-1}$ is $1$ and then $M_i =  M_{i-1} + Rx_i$ for any $x_i \in M_i\setminus M_{i-1}$. Set $N = M+ Rx_i$, then $N+ MM_i = M+Rx_i + M^2 + M^{n-i +1}S = (M+ M^{n-i+1} S) + Rx_i = M_{i-1} +Rx_i =  M_i$. Then the Nakayama Lemma shows that $N= M_i$.
\end{proof}

 \begin{lemma}\label{4.7} Let $(R,M)$ be a  local ring such that $R/M$ is infinite ring and $R\subset S$ a subintegral finite Gilmer extension.

 The following statements are equivalent:
 \begin{enumerate}
\item $MS/M$ is an uniserial $R$-module; 
\item $R\subseteq R+MS$ is a chained extension;
\item ${\mathrm L}_R(MS/M)=n-1$, where $n$ is the nilpotency index of $M/(R:~S)$;
\item $(R/I,M/I)$  is a SPIR, for $I:= (M :MS)= (R: (R+MS))$ and there exists some $x\in MS$ such that $MS=M+Rx$.
\end{enumerate}

If one of the above statement holds, then $I=((R:S):M))$.
\end{lemma}
  \begin{proof} As in a previous proof, we may first assume that $(R:S)=0$, since $(R:S)\subseteq I\subseteq M\subseteq MS$. We also may assume that $M\neq (R:S)$. Indeed, if $M= (R:S)$, the conditions (1), (2), (3) and (4) are trivially satisfied. In particular, this holds when $R$ is a field.

  Assume that $MS/M$ is an uniserial $R$-module and let $T\in [R, R+MS]$.  Then $(T,N)$ is a local Artinian ring and $R\subseteq T$ is subintegral. It follows that  $T/N\cong R/M$ and $T= R+N$ with $R\cap N = M$, and $N\in \llbracket M,MS\rrbracket$ where $\llbracket M, MS\rrbracket $ denotes the set of $R$-modules between $M$ and $MS$. Let $\varphi : [R,R+MS] \to \llbracket M,MS\rrbracket$ be the map defined by $T\mapsto N$. Then $\varphi$ is an injective map which is order preserving. It follows that $[R,R+MS]$ is a chain.
  
 We have (2) $\Leftrightarrow$(3) in view of \cite[Lemma 4.1]{Pic 6} since $R$ is not a field. 

Assume that $R\subset R+MS$ is a chained extension, and then FIP. We will denote by $\bar x$ a typical element  of $N:= MS/M\cong M(S/R)$. First observe that $I$ is the annihilator of the $R$-module $N =M_{n-1}/M = R\overline{x_{n-1}}$.  The map $\psi : R \to N$ defined by $a\mapsto a\overline{x_{n-1}}$ is  a surjective morphism of $R$-modules. Then $I =\ker \psi$ is an ideal of $R$ and $N\cong R/I$ as an $R$-module. We denote by $\overline{\psi}$ the induced isomorphism $R/\ker\psi \to N$. But by Lemma~\ref{4.6}, $M+ M^2S = M + Rx_{n-2}$, so that $(M+M^2S)/ M \cong R\overline{x_{n-2}}= MN=M \psi(R) =\psi(M) = \overline{\psi} (M/I)$. We deduce that $M/I = \overline {\psi}^{-1}(R\overline{x_{n-2}})=(R/I)\overline{\psi}^{-1}(\overline{x_{n-2}})$ is monogenic as an $R/I$-module, whence is a principal ideal of $R/I$. Since $R$ is local Artinian, we get that $R/I$ is a SPIR. And there exists some $x=x_{n-1}\in MS$ such that $MS=M+Rx$ by Lemma~\ref{4.6}.

 Conversely, assume that  $(R/I,M/I)$  is a SPIR, for $I:= (M :MS)= (R: (R+MS))$ and there exists some $x\in MS$ such that $MS=M+Rx$. The preceding proof shows that we have an isomorphism of $R$-modules $R/I\cong MS/M$. Since $R/I$ is a SPIR, its ideals are linearly ordered. As they are also its $R$-submodules, the $R$-submodules of $MS/M$ are linearly ordered. 
  
 Assume that  one of the above statement holds. Since $MS$ is the maximal ideal of $R+SM$, we get that $IMS\subseteq R$, so that $IM\subseteq (R:S)=0$, giving $IM=0$, and $I\subseteq (0:M)$.  But, $(0:M)MS=0$ gives $(0:M)=I$.
 
 Coming back to the general case, where $(R:S)\neq 0$, we obtain $I=((R:S):M)$. 
 \end{proof}
 
 \begin{example}\label{4.8} We are going to show that the two conditions of (4) of Lemma~\ref{4.7} are necessary. 
 Let $K$ be an infinite field. Set $R:=K[T]/(T^3)$ and $S:=K[T,Y]/(T^3,Y^3,T^2Y^2)$. Let $t$ (resp. $y$) be the class of $T$ (resp. $Y$) in $R$ (resp. $S$). Then,  $(R,M)$ is an infinite  SPIR, where $M=Rt$, and $R\subset S$ a subintegral finite extension, such that $R/(R:S)$ is an Artinian ring. In fact, $(R:S)=0$ and $M^3=0,\ M^2\neq 0$, so that the nilpotency index of $M$ is 3. Using notations given before Proposition~\ref{4.5} and in Lemma~\ref{4.6}, we get $M_2:=MS=St$, and $M_1:=M+M^2S=Rt+St^2$. We see easily that  $\{1,t,t^2\}$ is a basis of the $K$-vector space $R$ and  $\{1,t,t^2,y,ty,t^2y,y^2,ty^2\}$ is a basis of the $K$-vector space $S$. This gives that $M_2$ is a $K$-vector space with basis $\{t,t^2,ty,t^2y,ty^2\}$ and   $M_1$ is a $K$-vector space with basis $\{t,t^2,t^2y\}$, giving $\dim_K(M_2)=5$ and $\dim_K(M_1)=3$, so that $\dim_K(M_2/{M_1})=2={\mathrm L}_{R/M}(M_2/{M_1})={\mathrm L}_R(M_2/{M_1})$, since $MM_2\subseteq M_1$. It follows that $R_1\subseteq R_2$ is not a minimal ramified extension (\cite[Lemma 5.12]{DPP2}), and $R\subset R+MS$ is not a chained extension, because it has not FIP. So, the condition that $(R/I,M/I)$  is a SPIR is not sufficient  to guarantee  that $R\subset R+MS$ be  chained.
\end{example}
   
   \begin{theorem}\label{4.9} Let $(R,M)$ be a local ring such that $R/M$ is infinite and let $R\subset S$ be a finite and subintegral Gilmer extension.
     The following statements are equivalent:
\begin{enumerate}
\item $R\subset S$ has FIP;
\item $R\subset S$ is chained;
\item If $n>1$, then $[R,R_{n-1}] =  \{R_k \, | \, k=0,\ldots , n-1 \}$ and $|[R_{n-2},S]|\leq 4$. If $n=1$, then $|[R,S]|\leq 4$;

\item If $n>1$, then $MS/M$ is an uniserial $R$-module and $R_{n-2} \subset S$ is chained. \end{enumerate}
\end{theorem}
\begin{proof} We have already proved that (1) $\Leftrightarrow$ (2) $\Leftrightarrow$ (3).
Now (2) $\Leftrightarrow$ (4) by Lemma~\ref{4.7} and  \cite[Proposition 5.15]{DPP2} when $n>1$. 
\end{proof}

We note here that an integral Gilmer extension $R \subseteq S$ is spectrally injective if and only if  $ \substack{+\\ S}R= \substack{t\\ S}R$ \cite[Theorem 3.11]{Pic 1}. Moreover, if $R$ is local and $R/(R:S)$ is Artinian, then $R\subseteq S$ is spectrally injective if and only if $S$ is a local ring. Indeed, $\mathrm{D}_S((R:S)) \to \mathrm{D}_R((R:S))$ is a bijective map.

\begin{proposition}\label{4.10} Let  $R\subset S$ be a (finite)  local \'etale  Gilmer extension,  that is not seminormal and whose nucleus is an infinite Artinian ring.
Then $R\subset S$ has FIP if and only if $R\subset  \substack{+\\ S}R$ is chained.
\end{proposition}
\begin{proof}  We denote by $M$ the maximal ideal of $R$. We know that $R\subset S$ is finite and Gilmer, then integral and FCP. Denote by $\Sigma$ the seminormalization of $R$ in $S$. Then  $R\subset \Sigma $ is subintegral and also finite. Because of the spectral injectivity, $\Sigma$ is also the $t$-closure of $R\subset S$. Moreover $R/(R:\Sigma)$ is Artinian infinite and not a field. 
Indeed,  in case    $(R:\Sigma) = M$, we get that $M$ is a  maximal ideal of $\Sigma$, and then $R/M = \Sigma/M$ implies that $R=\Sigma$, an absurdity. Now $\Sigma/(\Sigma: S)$ is local and Artinian. It follows that the $t$-closed and unramified finite extension $\Sigma \subseteq S$ has FIP, thanks to Proposition~\ref{2.3}. Since $R\subseteq S$ has FIP if and only if $R\subset \Sigma$ and $\Sigma \subseteq S$ have FIP \cite[Theorem 5.9]{DPP2}, $R\subset S$ has FIP if and only if $R\subset \Sigma$ has FIP. To conclude, use Theorem~\ref{4.9}.
\end{proof}

The lacking parts of the canonical decomposition, when the base ring is local, are particular cases of results of Section 5, namely Proposition~\ref{5.5} and Corollary~\ref{5.6}.

\section{\'etale and seminormal extensions}

Lemma~\ref{2.1} shows that the seminormality of $R \subseteq R^n$ is equivalent to the  reduction of $R$. We are thus lead to examine \'etale extensions that are seminormal and what happens when $R\subseteq S$ is seminormal, a case we excluded in 
Proposition~\ref{4.10}.

 \begin{proposition}\label{5.2} Let $R\subseteq S$ be a 
   finite extension, whose nucleus  is zero-dimensional. Then $R\subseteq S$ is seminormal if and only if  
 $S/(R:S)$ is absolutely flat if and only if $S/(R:S)$ is reduced.
 \end{proposition}
 \begin{proof} 
   
 We can assume that $(R:S)=0$ and that $R$ is local with maximal ideal $M$. Indeed, seminormality and absolute flatness commute with localization. Moreover, $S/(R:S)$ is reduced if and only if it is locally reduced.  Then, $(R,M)$ is a zero-dimensional local ring and $S$ is a zero-dimensional semi-local ring. 
 
 Suppose that $R\subseteq S$ is seminormal. We get that $(R:S)=0=N_1 \cap \cdots \cap N_n = M$ where $\{N_1, \ldots,N_n\}$ is the set of all maximal ideals of $S$. In fact, $R\subseteq S$ is finite, so that there are finitely many maximal ideals in $S$ ($\mathrm{Max}(S)=\mathrm{Spec}(S)$ since $S$ is zero-dimensional). Then, $S\cong \prod_{i=1}^nS/N_i$ is a product of finitely many fields. It follows that $S$ is absolutely flat and reduced. 
 
 To show the converse, assume that $S$ is absolutely flat.  It is enough to observe that $R$ is a field since a reduced local zero-dimensional ring, because $S$ is a reduced zero-dimensional ring. Then, $0=(R:S)$ is semiprime in $S$, and   $0=(R:T)$ is semiprime in $T$ for each $T\in[R,S]$. It follows that $R\subseteq S$ is seminormal \cite[Lemma 4.8]{DPP2}. In particular, this holds if $S$ is reduced. 
  \end{proof}
     
   The next result generalizes 
    Proposition~\ref{5.2}
   
 \begin{proposition}\label{5.4} Let $ R \subseteq  S$ be an \'etale ring extension such that $R$ is Artinian. Then  $S/(R:S)$ is reduced if and only if $R\subseteq S$ is  seminormal and in this case $R\subseteq S$ has FIP.
Moreover, $|[R,S]|\leq\prod_{M\in\mathrm{MSupp}(S/R)}B_{n_M} $, where $n_M=\mathrm{rk}_{R/M}(S/MS)$.
 \end{proposition}
\begin{proof}  Observe that $R\subseteq S$ is finite  by Corollary~\ref{2.5}. Then $R\subseteq S$ is seminormal  if and only if $S/(R:S)$ is reduced by Proposition~\ref{5.2}. To show that $R\subseteq S$ has FIP,  we can assume that $R$ is a field by using Proposition~\ref{SUP}(2). Indeed, we may assume that $(R:S)=0$ and that $R$ is local. Now $R\subseteq S$ has property FIP by Theorem~\ref{ET}.  By \cite[Definition 1, Ch. 5, p.28]{Bki A}, there is a base change $R \to L$, where $L$ is a field extension of $R$, such that $L\otimes_R S \cong L^n$. Then $n$ is necessarily $ \mathrm{dim}_R(S)$.  The result follows from  Proposition~\ref{1.5} and \cite[Lemma 2.1]{DPP3}, since the base change $R\to L$ is faithfully flat.
  The last statement is a consequence of \cite[Theorem 3.6]{DPP2}.
\end{proof}

 We first offer two results, showing that the involved morphisms do be unramified when $R \subset S$ is seminormal.

\begin{proposition}\label{5.5}  Let $R\subset S$ be a seminormal and infra-integral finite Gilmer extension. Then $R\subset S$ is unramified and has FIP. 
\end{proposition}
\begin{proof} We can suppose that $(R,M)$ is a local ring since $|\mathrm{Supp}(S/R)|< \infty$. Because $R\subseteq S$ is seminormal,  $(R:S)$ is a semi-prime ideal in $S$.
It follows that $(R:S)= M$ and also $(R:S) = P_1\cap\cdots \cap P_n$, where $P_1,\ldots, P_n$ are the maximal ideals of $S$ containing $(R:S)$. Since $R\subseteq S$ is infra-integral,  $R/(R:S)\subseteq S/(R:S)$ identifies to $R/M \subseteq (R/M)^n$. Since this last extension is \'etale and has FIP, then $R\subseteq S$ is unramified by Lemma~\ref{Omega} and has FIP, because $\mathrm{Supp}(S/R) = \{M\}$.
\end{proof}

\begin{corollary}\label{5.6}  Let $R\subset S$ be a seminormal finite unramified Gilmer extension. Then $R\subseteq \substack{t\\ S}R$  and $\substack{t\\ S}R \subseteq S$ are unramified and have FIP. They are \'etale if $R \subset S$ is \'etale.
\end{corollary}
\begin{proof} The statement about $R\subset \substack{t\\ S}R$ follows from Proposition~\ref{5.5}. To prove the statement for $\substack{t\\ S}R \subset S$, it is enough to show that  this extension is a Gilmer extension, thanks to Proposition~\ref{2.3}. But the Gilmer condition follows from $(R:S)\subseteq (\substack{t\\ S}R :S)$
and the finiteness of $R\subseteq \substack{t\\ S}R $. For the last property, use  Lemma~\ref{Seppro}.
\end{proof}

\begin{theorem}\label{5.7}  Let $R\subseteq S$ be an \'etale extension. The following statements hold:

\begin{enumerate}

\item  $R\subseteq S$ has FIP,  is seminormal and finite  if and only if its nucleus is Artinian and reduced.

\item $R\subseteq S$ has FIP, is finite  and  and its nucleus is reduced if and only if $R\subseteq S$ is a seminormal Gilmer extension.
\end{enumerate}
\end{theorem}
\begin{proof}

(1) Assume that $R/(R:S)$ is Artinian and reduced. In that case $R/(R:S)\subseteq S/(R:S)$ is finite and \'etale. In view of Theorem~\ref{2.7}, this extension has FIP and is seminormal  and finite. The same properties hold for $R\subseteq S$.

Conversely, Assume that $R\subseteq S$ has FIP and is seminormal and finite. We know that $R\subseteq S$ is a Gilmer extension, because it is finite. Moreover, the seminormality of $R\subseteq S$ entails that $(R:S)$ is semi-prime in $S$ \cite[Lemma 4.8]{DPP2}, whence in $R$.

(2) If $R\subseteq S$ has FIP and is finite, then $R/(R:S)$ is Artinian, which gives by (1) that $R\subseteq S$ is seminormal when moreover, $R/(R:S)$ is reduced.

Conversely, the seminormality of $R\subseteq S$ entails that $R/(R:S)$ is reduced since $(R:S)$ is a radical ideal. It follows by (1) that $R\subseteq S$ has FIP when $R/(R:S)$ is Artinian.
\end{proof}

Consider the four conditions of  Theorem~\ref{5.7}: 

(a): finite FIP, (b): Gilmer, (c): seminormal, and (d): reduced nucleus.

 We proved that (a )$\land$ (c) $\Leftrightarrow$ (b) $\land$ (d), and  (a) $\land$ (d) $\Leftrightarrow$ (b) $\land$ (c).  In fact, (a) $\Rightarrow $ (b) (a finite FIP extension is Gilmer), and (c) $\Rightarrow$ (d) (a seminormal extension has a reduced nucleus). But ignoring these two results,  the examples below show that these four conditions   are logically independent. 

(b) $\not\Rightarrow$ (a) and (d) $\not\Rightarrow$ (a): Take $K\subset K[X]$, where $K$ is a field and $X$ an indeterminate.

(a) $\not\Rightarrow$ (c), (b) $\not\Rightarrow$ (c) and (d) $\not\Rightarrow$ (c): Take $R\subset S$ minimal ramified. 

(c) $\not\Rightarrow$ (b) and (d) $\not\Rightarrow$ (b): Let $R$ be a PID and set $S:=R_p$, where $p$ is a prime element.

(a) $\not\Rightarrow$ (d): Take $R\subset R^2$, where $R$ is a SPIR.

(b) $\not\Rightarrow$ (d): Let $K$ be a field. Take $R:=K[X]/(X^2)$ and $S:=R[Y]$ and $X,Y$  indeterminates.

(c) $\not\Rightarrow$ (a): Let $R:= \mathbb{Z}/2\mathbb{Z}$  and $\{X_i\}$ an infinite family of indeterminates. Set $S:=R[X_i]/(X_i^2-X_i,X_iX_j)$.  

As a consequence of the above theorem, we see that a seminormal 
integral \'etale extension has FIP if and only if it has FCP (use \cite[Theorem 4.2]{DPP2}). Moreover, it follows from Theorem~\ref{2.7} that a unramified extension $R\subseteq S$, whose conductor is an intersection of finitely many maximal ideals of $R$, has FIP. We intend now to look at unramified extensions that have FIP. To begin with we consider integral minimal extensions.

We  give three  lemmata, before giving another characterizations of seminormal  \'etale extensions having FIP.

\begin{lemma} \label{Artun} Let $R\to S$ be a finite ring morphism with finite separable residual extensions, such that $R$  is Artinian reduced and $S$ is reduced. Then $R\to S$ is \'etale. 
\end{lemma}
\begin{proof} Since $R$ is absolutely flat and Noetherian, we get that $R\to S$ is flat and $R\to S$ is of finite presentation. Let $P$ be a prime ideal of $R$, then $R_P$ is a field and then $S_P$ is Artinian, whence $S_P\cong K_1\timesÊ\cdots \times K_n$ a product of fields that are finite extensions of $\kappa (P)$. The hypotheses show that these field extensions are separable. Therefore, $0=\Omega(S_P|R_P) \cong \Omega(S|R)_P $ for each $P$ \cite[Proposition 11, p.34]{R}. It follows that $\Omega (S|R)=0$ and $R\to S$ is unramified, whence \'etale. 
\end{proof}

\begin{lemma}\label{unra} Let $R\subset S$ be a seminormal, integral extension which has FIP. Then any minimal morphism $A\subset B$ which appears in a decomposition into minimal morphism of $R\subseteq S$ is either inert or decomposed, whence unramified in case $R\subset S$ is \'etale.
\end{lemma}
\begin{proof} We observe that $R\subset S$ is finite because it has FIP.  Let $I$ be the conductor of $R\subset S$ and $M$ the conductor of $A\subset B$. By seminormality, $I$ is semi-prime in $S$ and is therefore a finite intersection of maximal ideals, because $R/I$ and $S/I$ are Artinian reduced. Observe that $I\subseteq M$ in $B$. Then $M$ is in $B$ a finite intersection of maximal ideals $M_1,\ldots, M_n$ because $M/I$ is an ideal of the absolutely flat ring $B/I$. It follows then that $A\subset B$ cannot be ramified, for if not there is a maximal ideal $N$ in $B$ such that $N^2\subseteq M \subset N$, and then $N\in \{ M_1, \ldots,  M_n\}$, an absurdity. The conclusion follows from Proposition~\ref{MOmega}
\end{proof}

Finally, we give a characterization of seminormal \'etale extensions that have FIP. We first rewrite a result of Ferrand-Olivier.

\begin{lemma}\label{5.1} Let $R\subseteq S$ be a ring extension, with conductor $\cap_{i=1}^nM_i,$

\noindent $ M_i\in\mathrm{Max}(R)$ for each $i$. Then $R_{M_i}$ is a field for each $i$ if and only if $R\subseteq S$ is flat.
\end{lemma}
\begin{proof}  As $\mathrm{Supp}(S/R) = \{M_1,\ldots,M_n \}$, if  $R_{M_i}$ is a field for each $i$, then $R\subseteq S$ is flat. Conversely if $R\subseteq S$ is flat, so is  $R_{M_i}\subseteq S_{M_i}$ with conductor $M_iR_{M_i}$ for each $i$. To complete the proof use \cite[Lemme 4.3.1]{FO}.
\end{proof}

\begin{theorem}\label{5.10} Let $R\subset S$ be an extension of finite presentation. The following statements are equivalent.

\begin{enumerate}

\item $R\subset S$ is a  seminormal  \'etale extension, having FIP;

\item  $R\subset S$ has finite separable residual field extensions and is finite,  $\mathrm{MSupp}(S/R)$ is finite,  $R':= R_{(R:S)} $ is a reduced  Artinian ring and $S':= S_{(R:S)}$ is reduced;

\item same conditions as in $\mathrm{(2)}$,  except that $S'$ is reduced  is replaced with $(R':S')= 0$.
\end{enumerate}

If either $\mathrm{(1)}$, $\mathrm{(2)}$ or $\mathrm{(3)}$  is verified, then each minimal morphism $A \subset B$ with conductor $\mathcal{C}$, appearing in the canonical decomposition of $R \subset S$ is such that $A_{(A:B)}=A_{\mathcal{C}}$ is a field.

\end{theorem}
\begin{proof} Suppose that (1) holds. We observe that  $R\subseteq S$ is finite and $\mathrm{Supp}(S/R)$ is finite because it has FIP.   Write a decomposition $R\subset \cdots \subset A \subset B \subset \cdots \subset S$ of $R\subset S$, where $A\subset B$ is minimal. Observe that $R\subset A$ is unramified  by Lemma~\ref{unra} because a composite of unramified ring morphism is unramified. Then apply inductively Lemma~\ref{Seppro}(3).  By using Lemma~\ref{5.1}, we get that $A_{\mathcal C}$ is a field for $\mathcal C= (A:B)$. In particular $R_M$ is a field  for  $M\in \mathrm{MSupp}(S/R)$. The ring $R'$ is  semi-local  with maximal ideals $M'_1, \ldots,M'_n$ that are the minimal prime ideals of $R'$.  Since $R'\to \prodÊ\, [ R'_{M'_i} \, | i= 1\ldots, n]$ is injective and its target is a product of fields, $R'$ is reduced and  is therefore  Artinian by the Chinese Remainder Theorem. In fact, $\dim(R')=0$, and $R'$ reduced give that $R'\cong\prod_{i=1}^nR'_{M'_i}$.  

By seminormality $(R':S')  $ is semiprime in $S'$ and then in $R'$. Also $(R':S' )\subseteq \mathrm{Rad}(R') =\mathrm{Nil}(R')= 0$. 
To complete the proof, we see that $R\subset S$ has finite separable residual extensions, because $R\subset S$ is \'etale. So (3) holds. 

If (3) holds then (2) holds because the absolute flatness of  $R'$  implies that  $R'\subseteq S'$ is seminormal. Therefore, its conductor is semiprime  in $S'$ and zero, so that $S'$ is reduced. 

To end, assume that (2) is valid.  We clearly  get that $R'\subseteq S'$ is seminormal and flat since $R'$ is absolutely flat. It follows that $R\subseteq S$ is seminormal and flat. Now $R \subseteq S$ is \'etale. It is enough to prove that $R'\subseteq S'$ is \'etale.
 
The conclusion follows from Lemma~\ref{Artun} and Remark~\ref{LOCAL}(2).  Finally,  $R'\subseteq S'$ has FIP by  Theorem~\ref{2.7}. We deduce that $R\subseteq S$ has FIP. Then, (1) holds.
\end{proof}

Observe that if $R\subset S$ is infra-integral, the separability hypotheses on residual extensions are verified.  Actually, we need only hypotheses of separability on residual extensions of $ \substack{t\\ S}R \subseteq S$. A minimal extension $A\subset B$ with conductor $M$ appearing in this last extension is inert, so that $A/M\subset B/M$ is the only residual extension of $A\subset B$  that may not  be separable. Since it is a  minimal extension of fields, it is either separable of purely inseparable. 

 Moreover,  the separability  of a residual extension $\kappa (P) \subseteq \kappa (Q)$ is verified if $(R:S) \nsubseteq Q$  since this extension is an isomorphism. 

  Also, reduction hypotheses are valid when $S$ is reduced.  They are also valid in case $R\subset S$  is \'etale
  and $R$ is reduced (Proposition~\ref{RED}). 
  
\begin{remark}\label{} Let $A$ be a ring, $p(X)\in A[X]$ and $B:=A[X]/(p(X))$, with $p(X)$ monic, so that $f:A\to B$ is faithfully flat. It is easy to show, using \cite[Lemma 2.6]{GP2}, that $f$ is infra-integral if and only if $p(X)$ splits in each $\kappa(P)[X]$ for $P\in\mathrm{Spec}(A)$, so that each fiber morphism  $\kappa(P)\to\kappa(P)\otimes_AB$ is of the form $\kappa(P)\to\kappa(P)^n$ for some integer $n$. It follows that  $f$ is \'etale if $f$ is infra-integral. Note that the conductor of $f$ is $0$. Therefore, when $f$ is infra-integral,  $f$ has FIP and is seminormal if and only if $A$ is Artinian reduced. The same result holds for a standard-\'etale algebra of the type $R\to R[X]/(p(X))$ where $p(X) \in R[X]$ is a monic polynomial whose derivative is invertible in $R[X]/(p(X))$.
\end{remark}

\end{document}